\documentclass[11pt,twoside]{article}
\usepackage{amsfonts,amssymb}
\usepackage{amsmath}
\usepackage{balance}

\usepackage{amssymb}
\usepackage{amsthm}
  \newtheorem{thm}{Theorem}[section]
  \newtheorem{lem}[thm]{Lemma}
  \newtheorem{prop}[thm]{Proposition}
   
  \newtheorem{cor}[thm]{Corollary}

  \theoremstyle{definition}
  \newtheorem{defn}{Definition}[section]

  \theoremstyle{remark}
  \newtheorem*{rem}{Remark}

\newcommand{\abs}[1]{\left|#1\right|}

\newcommand{\aleq}{\lesssim}

\title{Quasiregular values from generalized manifold with controlled geometry}
\author{Deguang Zhong \\Institute of Applied Mathematics, Shenzhen Polytechnic University,\\
Shenzhen, Guangdong, 518055, P. R. China\\}
\date{}

\chardef\bslchar=`\\ 

\providecommand{\qedsymbol}{\leavevmode
  \hbox to.77778em{%
  \hfil\vrule
  \vbox to.675em{\hrule width.6em\vfil\hrule}%
  \vrule\hfil}}

\catcode`\|=0
\begingroup \catcode`\>=13 
\gdef\?#1>{{\normalfont$\langle$\textit{#1}$\rangle$}}
\gdef\0{\relax}
\endgroup
\def\<#1>{{\normalfont$\langle$\textit{#1}$\rangle$}}

\def\latex/{{\protect\LaTeX}}

\usepackage{url}
\usepackage[breaklinks]{hyperref}

\begin{document}
\maketitle
\begin{abstract}
The main aim of this paper is to establish the Reshetnyak's theorem for quasiregualr values from generalized $n$-manifold with suitable controlled geometry to Euclidean space $\mathbb{R}^{n}.$
 This generalizes a previous result due to Kangasniemi and  Onninen on the setting of Euclidean space 
[A single-point Reshetnyak’s theorem, Trans. Amer. Math. Soc., $\mathbf{378}$(2025): 3105-3128].  
\end{abstract}
\begingroup
\small
\tableofcontents
\endgroup


\newpage 

\section{Introduction}
In the 1960s \cite{Res67}, Yu. G. Reshetnyak  introduced the following concept of quasiregular maps or maps of bounded distortion in Euclidean space. 
  \begin{defn}{\rm\cite{Res67}}\label{01900qs1}
  For a given constant $K\geq1$ and $n\geq2.$ Suppose that $\Omega\subset\mathbb{R}^{n}$  is an open connected set. Then, a mapping $f:\Omega\rightarrow\mathbb{R}^{n}$ is called quasiregular or a mapping of bounded distortion if it belongs to $W_{loc}^{1,n}(\Omega, \mathbb{R}^{n})$ and satisfies the following inequality 
  $$\vert Df(x)\vert^{n}\leq KJ_{f}(x)$$
  for almost every $x\in\Omega.$
  \end{defn}
If the constant $K$ in the Definition \ref{01900qs1} is replaced by a measure function $K:\Omega\rightarrow[1,\infty),$  then it leads to the definition of finite distortion mappings; c.f. \cite{HK14, IM01}.

Within the framework of nonlinear elasticity, investigating the topological features of finite distortion mappings $f$ in $\Omega\subset\mathbb{R}^{n}$ with distortion $K$ is of crucial importance. A  deep theorem of Reshetnyak \cite{Res67} states that each (locally non-constant) quasiregular map in a domain of $\mathbb{R}^{n}$ is discrete, open, and sense-preserving. Here we say a mapping $f:\Omega\rightarrow\mathbb{R}^{n}$ is discrete if for every $y_{0}\in\mathbb{R}^{n},$ the set $f^{-1}\{y_{0}\}$ consists of isolated points. A map $f:\Omega\rightarrow\mathbb{R}^{n}$  is said to be open if $f(U)$ is open set whenever $U$ is any open set of $\Omega.$ A continuous map is said to be sense-preserving if ${\deg}(y,f,D)>0$ whenever $D$ is a domain which compactly contained in $\Omega$ and $y\in f(D)\setminus f(\partial D).$ J. M. Ball \cite{Bal77,Bal81} showed that the quasiregularity is too strong for some problems in nonlinear elasticity. Hence, it is natural to consider the case for $K\in L_{loc}^{p}$ for some $p>0.$ For existing results on this problem in Euclidean spaces, we refer the reader to  \cite{Bal81,IS93,HK93,VM98,Bjö,HM02,HR13,KOZ26}. For the non-Euclidean setting, we refer to \cite{HH97,Zap08,Kir16,BV23,MR25,Zho251}. 
In particular, recently Kangasniemi and Onninen  \cite{KO222} has extended Reshetnyak's theorem to the case of quasiregular values in Euclidean spaces. Namely, they showed that a non-constant map $f\in W_{loc}^{1,n}(\Omega,\mathbb{R}^{n})$ satisfying the inequality
$$\vert Df(x)\vert^{n}\leq KJ_{f}(x)+\Sigma(x)\vert f(x)-y_{0}\vert^{n},$$
where $K\geq1$ is a constant, $y_{0}\in \mathbb{R}^{n}$ and $\Sigma\in L_{loc}^{1+\varepsilon}(\Omega)$ for some $\varepsilon>0,$ then $f^{-1}\{y_{0}\}$ is discrete, the local index $i(x,f)$ is positive in $f^{-1}\{y_{0}\}$, and every neighborhood of every point of $f^{-1}\{y_{0}\}$ is mapped to a neighborhood of $y_{0}.$ 

The main goal of this paper is to generalizes the above result of Kangasniemi and  Onninen \cite{KO222} for quasiregualr values from generalized $n$-manifold with suitable controlled geometry to Euclidean space $\mathbb{R}^{n}.$ It is read as follows.
 \begin{thm}\label{00990099199}
Let $K\geq1$ be a constant,  $y_{0}\in\mathbb{R}^{n},$ $\mathcal{S}\subset \mathbb{R}^{m}$ be a generalized $n$-manifolds with controlled geometry in the sense of subsection \ref{adjhaka} and 
$\Omega\subset\mathcal{S}$ be a domain. Suppose that $f\in N_{loc}^{1,n}(\Omega,\mathbb{R}^{n})$ is an nonconstant mapping  satisfying  the following inequality
\begin{equation}\label{---1}
\|apDf(x)\|^{n}\leq KJ_{f}(x)+\Sigma(x)\vert f-y_{0}\vert^{n}
\end{equation}
for almost every $x\in\Omega.$ If $\Sigma\in L_{loc}^{p}(\Omega)$ for some $p>1,$  then $f^{-1}\{y_{0}\}$ is discrete, the local index $i(x,f)$ is positive for every $x\in f^{-1}\{y_{0}\}$, and for every neighborhood $V$ of every  $x\in f^{-1}\{y_{0}\},$  there has $y_{0}\in {\rm int}f(V).$
\end{thm}
In the rest of this paper are arranged is as follows: In Sect. \ref{sec2222}, we will recall some notations and well known results on Newtonian spaces on the setting of metric measure space. We also  explain some assumptions posed on the generalized $n$-manifold with suitable controlled geometry and discuss the definition of  quasirelular values or values of finite distortion on generalized $n$-manifold. In Sect. \ref{sec255555},
we will establish the locally H\"{o}lder continuity for some subclass of generalized finite distortion maps on generalized $n$-manifold. Some corollaries of this locally H\"{o}lder continuity are also given.
  In  Section \ref{sec3333}, we mainly establish the totally disconnected for values of finite distortion. In this section, we also discuss some corollaries on this totally disconnected.
  In Section \ref{schjcsjb}, we give the proof of Reshetnyak's theorem for quasiregular values from generalized $n$-manifold with suitable controlled geometry to Euclidean space $\mathbb{R}^{n}.$

\section{Preliminary}\label{sec2222}
In this section, we will recall some facts which we need in this paper. It should be pointed out that  $C(k_{1},k_{2},\ldots,k_{m})>0$ in this paper will be a constant that only depends on $k_{1},k_{2},\ldots,k_{m}$ but may vary from line to line. For any metric space $X,$ the distance between any two points $x,y\in X$ is donated by $\vert x-y\vert.$ Also,  we say $A\aleq B$ if there exists a constant $c>0,$ such that $A\leq cB.$
The symbol $f_{U}$ is defined as  $f_{U}:=\frac{1}{\mathcal{H}^{n}(U)}\int_{U} f d\mathcal{H}^{n}.$

\subsection{Newtonian spaces}
The objects of this paper is the generalized $n$-manifold, which in general does not admit a smooth structure. Consequently, the theory of Sobolev spaces in the smooth setting is no longer applicable. Instead, we employ the theory of Newtonian spaces on metric measure spaces. Newtonian spaces, introduced by N. Shanmugalingam in \cite{Sha00}, provide an analogue of Sobolev spaces in the metric context. This theory, which relies on the concept of upper gradients, is applicable to metric spaces that are rich enough to contain sufficiently many rectifiable paths. The main references  in this subsection are taken from \cite{Sha00,HKST01,HKST15}.
\begin{defn}
Suppose that $(X,\mu)$ is a metric measure space, where $\mu$ is a Borel regular measure. Let $\Gamma$ be  collection of paths in $X$ and $p\geq1.$ Then a Borel function 
$\rho:X\rightarrow[0,\infty]$ is said to be admissible for $\Gamma$  if the following inequality 
$$\int_{\gamma}\rho ds\geq1$$
holds for every locally rectifiable $\gamma\in\Gamma.$ In this case, we donate  $\rho\in{\rm Adm}\Gamma.$ Then, the $p$-modulus of $\Gamma$  is defined by
$${\rm Mod}_{p}\Gamma:=\inf\int_{X}\rho(x)^{p}d\mu,$$
where the infimum is taken over all admissible functions $\rho\in{\rm Adm}\Gamma.$
\end{defn}

\begin{defn}
Suppose that $f:X\rightarrow \overline{\mathbb{R}}$ is a mapping from metric spaces $X$ to $\overline{\mathbb{R}}.$ Then, a Borel function $g:X\rightarrow[0,\infty]$ is said to be an upper gradient of $f$ if the following inequality 
\begin{equation}\label{11}
\vert f(x)-f(y) \vert\leq\int_{\gamma}gds
\end{equation}
holds for every locally rectifiable path $\gamma:[0,1]\rightarrow X$ with $\gamma(0)=x$ and $\gamma(1)=y.$ Moreover, we say $g$ is a weak upper gradient of $f$ if there exists a path family
$\Gamma_{0},$ such that ${\rm Mod}_{p}(\Gamma_{0})=0$ and  the inequality  (\ref{11}) holds for every $\gamma \notin\Gamma_{0}.$
\end{defn}

\begin{defn}(Newtonian spaces)\cite{Sha00}
Suppose that $(X,\mu)$ is a metric space and $p\geq1.$ Then, we say a real-value mapping $f:X\rightarrow\overline{\mathbb{R}}$ belongs to $N^{1,p}(X)$ if $f\in L^{p}(X)$ and has 
a $p$-integrable $p$-weak upper gradient. We say a mapping $f:X\rightarrow\mathbb{R}^{m}$ belongs to $N^{1,p}(X,\mathbb{R}^{m})$ if its every component function  belongs to $N^{1,p}(X).$
A real-value mapping $f\in N_{loc}^{1,p}(X)$ if $f\in N^{1,p}(D)$ for every compact $D\subset X.$ Further, we say $N_{loc}^{1,p}(X,\mathbb{R}^{m})$ if $f\in N^{1,p}(D,\mathbb{R}^{m})$ for every compact $D\subset X.$
\end{defn}

\begin{defn}(Weak Sobolev-Poincar\'{e} inequality)\cite{HKST01,HKST15}
Let $1\leq p<q$ and $(X,\mu)$ be a metric measure space. Then, the space $X$ is said to support a weak Sobolev-Poincar\'{e} inequality if there exist constants $C>0$ and $\lambda\geq1,$ such that the following inequality 
\begin{equation}\label{djaadlla}
\left(\int_{ B(x,r)}\!\!\!\!\!\!\!\!\!\!\!\!\!\!\!\!\!\!\!\!\!\; {}-{} \,\;\;\,\,\,\;\;\abs{u-u_{B(x,r)}}^{q}\;d\mathcal{H}^{n}\right)^{1/q}\leq C{\rm diam}(B)
\left(\int_{ B(x,\lambda r)}\!\!\!\!\!\!\!\!\!\!\!\!\!\!\!\!\!\!\!\!\!\!\!\; {}-{} \,\;\,\;\;\,\,\,\;\;g^{p}\;d\mathcal{H}^{n}\right)^{1/p}
\end{equation}
for every $B(x,r)\subset X$ and for every integrable functions $u$ in $B(x,r)$ and for all $p$-weak upper gradients $g$ of $u.$
\end{defn}

The case for $q=1$ in the inequality (\ref{djaadlla}) is called the weak $p$-Poincar\'{e} inequality. 

\subsection{Generalized $n$-manifolds with controlled geometry}
In this paper, some  conditions on controlled geometry are posed on the metric space which we study. We first recall some related concepts. 

\subsubsection{Ahlfors $Q$-regular}
\begin{defn}\cite{HS02}
Suppose that $Q>0.$ Then, $\mathcal{S}\subset \mathbb{R}^{m}$ endowed with a Borel measure $\mu$ is said to be Ahlfors $Q$-regular if there exists a constant $C\geq1,$ such that the follow inequalities 
\begin{equation}\label{shiojcks}
\frac{1}{C}r^{Q}\leq \mu(B(x,r))\leq C r^{Q}
\end{equation}
holds for every $B(x,r)\subset \mathcal{S},$ where $r\leq{\rm diam}(\mathcal{S}).$
\end{defn}

It is well known that the measure $\mu$ satisfies (\ref{shiojcks}) can be replaced by the Hausdorff measure $\mathcal{H}^{n};$ see e.g. \cite[Lemma C.3]{Sem96}. Hence, the measure we always used on $\mathcal{S}\subset \mathbb{R}^{m}$ in this paper is 
the Hausdorff measure $\mathcal{H}^{n}.$
\subsubsection{Rectifiable sets}
Rectifiable sets constitute one of the most important classes of sets in geometric measure theory. It is defined by the following. 
\begin{defn}{\rm\cite{Fed69}}
Suppose that $\mathcal{S}\subset \mathbb{R}^{m}$ is $\mathcal{H}^{n}$-measurable, where $m\geq n\geq2.$ Then $\mathcal{S}$ is said to be $n$-rectifiable if there exist Lipschitz mappings $\varphi_{j}:
\mathbb{R}^{n}\rightarrow\mathbb{R}^{m},$ such that 
$$\mathcal{H}^{n}(\mathcal{S}\setminus\cup_{j=1}^{\infty}\varphi_{j}(\mathbb{R}^{n}))=0.$$
\end{defn}
A well known result on the property for rectifiable sets is stated as follows.
\begin{prop}{\rm\cite[Theorem 15.19]{Mat95}}\label{sfbgdnbd}
Suppose that $\mathcal{S}\subset \mathbb{R}^{m}$ is $n$-rectifiable and $\mathcal{H}^{n}$-measurable. Then for $\mathcal{H}^{n}$-almost all $x\in\mathcal{S},$ there is a unique approximate tangent 
$n$-plane $apTan(x,\mathcal{S})$ for $x$ at $\mathcal{S}.$
\end{prop}
Suppose that $U$ is an open subset of $\mathcal{S}.$ Then, according to Proposition \ref{sfbgdnbd}, we see for almost every $x\subset U,$ there exists a unique $n$-dimensional plane  $apTan(x,U).$
The symbol $TU$ stands for the collection of all these planes associated to points in $U.$  

\subsubsection{Generalized $n$-manifolds}
Assume that the space $X$ is locally compact, separable, connected,  locally connected. Then, from \cite[Definition 1.1]{HS02}, we say that the space $X$ is a $\mathbf{cohomology}$ $n$-$\mathbf{manifold}$ if its topological dimension is at most $n$ and 
 local cohomology groups satisfying $H_{c}^{n}(V)=n$ and  $H_{c}^{n-1}(V)=0.$ Here $V$ is an open neighborhood of every $x\in X$ which contained in every neighborhood $U$ of $x,$ and $H_{c}^{*}$ stands for the Alexander-Spanier cohomology groups with compact supports.
\begin{defn}\cite[Definition 1.6]{HS02}(Generalized $n$-manifolds)\label{ajojlaxax}
Suppose that  $n\geq2.$ Then, a space $X$ is said to be a generalized $n$-manifold if it is a finite-dimensional cohomology $n$-manifold.
\end{defn}
 
 \subsubsection{Orientable and metrically orientable}
 The following concept on oriented generalized $n$-manifolds can be found from  \cite[Section 2]{HS02}.
\begin{defn}\cite{HS02}(Oriented generalized $n$-manifolds)
Suppose that the space $X$ is a generalized $n$-manifold. Then $X$ is said to be orientable if $H_{c}^{n}(X)\approxeq\mathbb{Z},$ and a generator $g_{X}$ is called an orientation of $X.$ The pair $(X, g_{X})$ is called an oriented generalized $n$-manifold. 
\end{defn}
 Since for every open neighborhood $V$ of every $x,$ the following standard homomorphism
 \begin{equation}\label{djhaopjkdcja}
 H_{c}^{n}(W)\rightarrow H_{c}^{n}(V)
 \end{equation} 
  is a surjection. Here $W\subset V$ is an open neighborhood of $x.$ Hence, if we suppose that $\mathcal{S}$ is an oriented, generalized $n$-manifold and  $U$ is an open and  connected subset of $\mathcal{S}.$ Then the mapping in (\ref{djhaopjkdcja})  induced a fixed orientation of
  $U.$ In addition,  from  \cite[Lemma 3.2.25]{Fed69}, we see that there is a measurable choice of orientation $g_{x}$  on each $apTan(x,U)$ in $TU,$ and we call this is an orientation of the tangent bundle $TU.$
Now, from  \cite[Remark 2.14]{Kir14} it gives that if an open connected neighborhood $D$ of $x$ in $U$ is small enough, then the projection $$\pi_{x}: \mathbb{R}^{m}\rightarrow apTan(x,U)
\setminus\pi_{x}(\partial D)$$ satisfying $x\notin\pi(\partial D).$ Thus, if $Z$ is the $x$-component of $apTan(x,U)\setminus\pi_{x}(\partial D),$ then we have
\begin{equation}\label{djaihd}
H_{c}^{n}(apTan(x,U))\leftarrow H_{c}^{n}(Z)\stackrel{\pi_{x}^{*}}{\longrightarrow}H_{c}^{n}(\pi_{x}^{-1}(Z)\cap D)\rightarrow H_{c}^{n}(D)\rightarrow H_{c}^{n}(U).
\end{equation}
 Here, the symbol $\pi_{x}^{*}$ stands for  the homomorphism induced by $\pi_{x},$ and unnamed arrows represent canonical isomorphisms induced by embeddings. Now,  we can give the following concepts.
 \begin{defn}{\rm\cite{HK11,HR02,HS02}}(Metrically orientable and  locally metrically orientable)
Suppose that $U$ is the open subset of $\mathcal{S}$ discussed above. Then, $U$ is said to be  metrically orientable if there is an orientation of tangent bundle $TU$ such that for almost every
 $x\in U$ $g_{x}$ is mapped to $g_{U}$ under the mapping represented above. We says $U$ is metrically oriented if such an orientation of $TU$ is chosen.  We say that $\mathcal{S}$ is locally metrically 
 orientable if every point in $\mathcal{S}$  has a neighborhood that is metrically orientable. 
\end{defn}

 \subsubsection{Linearly locally contractible}
The following concept on linearly locally contractible can also be found in \cite{HK11,HR02,HS02}.
\begin{defn}
For a metric space $\mathcal{S},$ if there is a constant $C\geq1,$ such that for every compact $D\subset\mathcal{S},$ there exists a positive number $r_{D},$ such that $B(x,r)$ is contractible 
in $B(x,Cr)$ for $r\leq r_{D}$ and $x\in D.$ Then, we say $\mathcal{S}$ is linearly locally contractible.
\end{defn}
 
 \subsubsection{Local degree and local index}
\begin{defn}\cite{HS02}(Local degree)
Suppose that $f: X\rightarrow Y$ is continuous mapping, where $X$ and $Y$ are two oriented generalized $n$-manifolds. For each relatively compact domain $D$ in $X$  and for each component $V$ of 
$Y\setminus f(\partial D),$ the following composition of two maps
$$H_{c}^{n}(V)\rightarrow H_{c}^{n}(f^{-1}(V)\cap D)\rightarrow H_{c}^{n}(D)$$
gives the mapping 
$$g_{V}\mapsto m\cdot g_{D}.$$ 
 Here, the integer $m$ is called the local degree of $f$ at a point $y\in V$ with respect to $D,$ and is denoted by $m={\rm deg}(f,D,y).$ 
\end{defn}

\begin{defn}\cite{KO222}(Local index)
 Let $\Omega$ be a connected domain, $y_{0}\in\mathbb{G}$ and   $f:\Omega\rightarrow\mathbb{G}$ is continuous. Suppose that $x_{0}\in f^{-1}\{y_{0}\}.$ If $V_{1}$ and $V_{2}$ are two neighborhoods of $x_{0}$ such that $\overline{V_{1}}\cap f^{-1}\{y_{0}\}=\overline{V_{2}}\cap f^{-1}\{y_{0}\}=\{x_{0}\},$ then $y_{0}\notin f(V_{i}), i=1,2.$ Hence, it is meaningful to define  ${\rm deg}(f,V_{i},y_{0}),i=1,2.$ In fact, it well known that ${\rm deg}(f,V_{1},y_{0})={\rm deg}(f,V_{2},y_{0})={\rm deg}(f,V_{1}\cup V_{2},y_{0}).$ Therefore, if there exist a neighborhood $V$ of $x_{0},$ such that 
$\overline{V}\cap f^{-1}\{y_{0}\}=\{x_{0}\},$ then the same value of ${\rm deg}(f,V,y_{0})$ is regardless of the choice of $V.$ We call this value the local index of $f$ at $x_{0}$ and  record it as $i(x_{0},f).$
\end{defn}

 \subsubsection{Generalized $n$-manifolds with controlled geometry}\label{adjhaka}

By an oriented generalized $n$-manifold $\mathcal{S}\subset \mathbb{R}^{m}$  with controlled  geometry we mean an oriented  generalized $n$-manifold $\mathcal{S}$ equipped with a metric for which

(1) $\mathcal{S}$ is $n$-rectifiable;

(2) $\mathcal{S}$ is Ahlfors $n$-regular;

(3) $\mathcal{S}$ is linearly locally contractible.

This class of spaces fits into the framework of controlled geometry in the sense of Heinonen–Koskela \cite{HK98}. Actually,  a deep theorem of Semmes \cite{Sem96} states that if a metric measure space $\mathcal{S}\subset \mathbb{R}^{m}$ is Ahlfors $Q$-regular and  has cohomology modules and is linearly locally contractible, then 
$\mathcal{S}$  supports a weak $1$-Poincar\'{e} inequality, and hence a weak $p$-Poincar\'{e} inequality for all $p\geq1.$ Together with the local homological properties of a generalized manifold, these conditions ensure that analytic and topological tools—including proper maps, degrees, and Sobolev regularity—are well-defined and stable.

In addition, it should be noted that as $\mathcal{S}$  supports  a weak $p$-Poincar\'{e} inequality for all $p\geq1,$   it implies from \cite[Theorem 8.3.2 and Corollary 9.1.36]{HKST15} that $\mathcal{S}\subset \mathbb{R}^{m}$  supports the following Sobolev-Poincar\'{e} inequality
\begin{equation}\label{asddfvddjaadlla}
\left(\int_{ B(x,r)}\!\!\!\!\!\!\!\!\!\!\!\!\!\!\!\!\!\!\!\!\!\; {}-{} \,\;\;\,\,\,\;\;\abs{u-u_{B(x,r)}}^{\frac{np}{n-p}}\;d\mathcal{H}^{n}\right)^{\frac{n-p}{np}}\leq C{\rm diam}(B)
\left(\int_{ B(x,\lambda r)}\!\!\!\!\!\!\!\!\!\!\!\!\!\!\!\!\!\!\!\!\!\!\!\; {}-{} \,\;\,\;\;\,\,\,\;\;\abs{g}^{p}\;d\mathcal{H}^{n}\right)^{1/p}
\end{equation}
for every $B(x,r)\subset \mathcal{S}$ with $r\leq {\rm diam}\mathcal{S}$ and for every integrable functions $u$ in $B(x,r)$ and for all upper gradients $g$ of $u.$ Combing this with the density of Lipschitz 
mappings of $f=(f_{1,},\ldots,f_{j},\ldots,f_{n})\in N_{loc}^{1,p}(\mathcal{S},\mathbb{R}^{n})$ and \cite[Lemma 2.29]{Kir16}, we obtain
\begin{equation}\label{asddfvddjaadlla}
\left(\int_{ B(x,r)}\!\!\!\!\!\!\!\!\!\!\!\!\!\!\!\!\!\!\!\!\!\; {}-{} \,\;\;\,\,\,\;\;\abs{f_{j}-(f_{j})_{B(x,r)}}^{\frac{np}{n-p}}\;d\mathcal{H}^{n}\right)^{\frac{n-p}{np}}\leq C{\rm diam}(B)
\left(\int_{ B(x,\lambda r)}\!\!\!\!\!\!\!\!\!\!\!\!\!\!\!\!\!\!\!\!\!\!\!\; {}-{} \,\;\,\;\;\,\,\,\;\;\abs{apDf}^{p}\;d\mathcal{H}^{n}\right)^{1/p}
\end{equation}
for every $B(x,r)\subset \mathcal{S}$ with $r\leq {\rm diam}\mathcal{S}$ and for every $j\in\{1,2,\ldots,n\}.$

\subsection{Quasiregular values and values of finite distortion}
  As mentioned in the introduction, Yu. G. Reshetnyak  \cite{Res67} introduced the following concept of quasiregular maps or maps of bounded distortion in Euclidean space in the 1960s. 
  \begin{defn}{\rm\cite{Res67}}\label{019001}
  For a given constant $K\geq1$ and $n\geq2.$ Suppose that $\Omega\subset\mathbb{R}^{n}$  is an open connected set. Then, a mapping $f:\Omega\rightarrow\mathbb{R}^{n}$ is called quasiregular or a mapping of bounded distortion if it belongs to $W_{loc}^{1,n}(\Omega, \mathbb{R}^{n})$ and satisfies the following inequality 
  $$\vert Df(x)\vert^{n}\leq KJ_{f}(x)$$
  for almost every $x\in\Omega.$
  \end{defn}
 For further properties and generalizations of quasiregular mappings, see \cite{Ric93,HK14, IM01, Guo15, GGWX25,GW16}. For a given constant $K\geq1.$ Suppose that $\Omega\subset\mathbb{R}^{n},n\geq2,$  is a domain  and $\Sigma:\Omega\rightarrow[0,\infty)$ is a  measure function. Then, a mapping $f:\Omega\rightarrow\mathbb{R}^{n}$ belongs to the class of Sobolev space $W_{loc}^{1,n}(\Omega, \mathbb{R}^{n})$ and satisfies the following inequality 
 \begin{equation}\label{666y1}
\vert Df(x)\vert^{n}\leq KJ_{f}(x)+\Sigma(x)\vert f(x)\vert^{n}
\end{equation}
  for almost every $x\in\Omega,$ was introduced by Astala, Iwaniec and Martin \cite[Section 8.5]{ATM09}. The inequality  (\ref{666y1}) was called the heterogeneous distortion inequality in \cite{KO221}.
For the properties and applications for mappings relative to inequality (\ref{666y1}), we refer the reader to \cite{Nir53,FS58,Har58,Sim77,AP06}.  Recently,  Kangasniemi and  Onninen \cite{KO222} give the following concept of quasiregular values.
 \begin{defn}\cite[Definition 1.1]{KO222} 
  For a given constant $K\geq1.$ Suppose that $\Omega\subset\mathbb{R}^{n}$  is a domain, $y_{0}\in\mathbb{R}^{n}$ and $\Sigma\in L_{loc}^{p}(\Omega)$ for some $p>1.$  Then, a mapping $f:\Omega\rightarrow\mathbb{R}^{n}$ has a $(K,\Sigma)$-quasiregular value at $y_{0}$ if it belongs to $W_{loc}^{1,n}(\Omega, \mathbb{R}^{n})$ and satisfies the following inequality
  \begin{equation}
  \vert Df(x)\vert^{n}\leq KJ_{f}(x)+\Sigma(x)\vert f(x)-y_{0}\vert^{n}
  \end{equation}
  for almost every $x\in\Omega.$
  \end{defn}
We refer the reader to \cite{KO222,KO241, KO242} for the properties, for example the H\"{o}lder continuity, Liouville theorem, Reshetnyak’s theorem, Rickman’s Picard Theorem,  linear distortion and rescaling properties,  of  quasiregular values established by  Kangasniemi and  Onninen. More recently, Dole\v{z}alov\'{a},  Kangasniemi and Onninen \cite{DKO24} introduced and established the continuity for the Sobolev maps $f\in W_{loc}^{1,n}(\Omega,\mathbb{R}^{n})$ in Euclidean space which satisfying the following inequality
\begin{equation}\label{000032dcsc}
\vert Df(x)\vert^{n}\leq K(x) J_{f}(x)+\Sigma(x)
\end{equation}
for almost everywhere $x\in\Omega.$ Here $K:\Omega\rightarrow[1,\infty)$ and $\Sigma:\Omega\rightarrow[0,\infty)$ are the measurable functions. The author \cite{Zho252} generalized this kinds of mappings satisfying inequality (\ref{000032dcsc}) to generalized finite distortion curves from Euclidean domain to Riemannian manifolds.  Based on the developments for quasiregular maps and its generalizations discussed above, the concept for analogous mappings satisfying inequality (\ref{000032dcsc}) from generalized $n$-manifolds with controlled geometry to Euclidean space is natural to give; see Definition \ref{df1}. 
\begin{defn}
Suppose that $f\in N_{loc}^{1,n}(\mathcal{S},\mathbb{R}^{n}).$ Then, by a result in \cite[Corollary 2.17]{Kir14}, we see that $f$ is approximately differentiable $\mathcal{H}^{n}$-almost everywhere. This 
approximate derivative ${\rm ap}Df(x):apTan(x,\mathcal{S})-x\rightarrow \mathbb{R}^{n}$ is a linear mapping from the shifted approximate tangent plane $apTan(x,\mathcal{S});$ c.f. \cite[Definition 3.2.16]{Fed69}.
Its Jacobian determinatant is defined by $J_{f}:={\rm det\,ap}Df(x).$
\end{defn}

\begin{defn}\label{df1}
We say that $f\in N_{loc}^{1,n}(\mathcal{S},\mathbb{R}^{n})$ is a generalized finite distortion map if it satisfies the following inequality
\begin{equation}
\| {\rm ap}Df(x)\|^{n}\leq K(x)J_{f}(x)+\Sigma(x)
\end{equation}
for almost everywhere $x\in\Omega.$ Here $K:\Omega\rightarrow[1,\infty)$ and $\Sigma:\Omega\rightarrow[0,\infty)$ are the measurable functions, and $\| {\rm ap}Df(x)\|$ is defined by 
$\| {\rm ap}Df(x)\|:=\sup_{\vert y\vert=1}\vert apDf(x)y\vert.$
\end{defn}

 The following concept of values of finite distortion is analogical to the one in the setting of Euclidean space given  by Kangasniemi and Onninen \cite{KO25}.
\begin{defn}\label{df2}
Suppose that $f\in N_{loc}^{1,n}(\mathcal{S},\mathbb{R}^{n})$ and $y_{0}\in\mathbb{R}^{n}.$  Then $f$ is said to be having a $(K, \Sigma)$-finite distortion value at $y_{0}$ if it satisfies the following inequality
\begin{equation}\label{df2111111}
\| {\rm ap}Df(x)\|^{n}\leq K(x)J_{f}(x)+\Sigma(x)\vert f(x)-y_{0}\vert^{n}
\end{equation}
for almost everywhere $x\in\Omega.$ Here $K:\Omega\rightarrow[1,\infty)$ and $\Sigma:\Omega\rightarrow[0,\infty)$ are  measurable functions.
\end{defn}

\begin{rem} 
It is worth noting that in Definitions \ref{df1} and \ref{df2}, it still allows for $\Omega$ to have regions where $J_{f}$  is negative.
\end{rem}

 \section{H\"{o}lder regularity for generalized finite distortion and its applications}\label{sec255555}

\subsection{H\"{o}lder regularity for generalized finite distortion mappings}
In this section, we will establish the H\"{o}lder regularity for generalized finite distortion. In order to do that, we first need the following lemma which obtained by Kirsil\"{a} \cite[Theorem 3.1]{Kir16}.
\begin{lem}{\rm \cite[Theorem 3.1]{Kir16}}\label{Kir16}
  Suppose that $f\in N_{loc}^{1,n}(\mathcal{S},\mathbb{R}^{n}),$ $\Omega\subset\mathcal{S}$ is a relatively compact open set and $\eta:\Omega\rightarrow \mathbb{R}$ is a Lipschitz mapping with 
  ${\rm spt}\eta\subset\Omega.$ Then the following equation
  \begin{equation}
  \int_{\Omega}\eta J_{f}d\mathcal{H}^{n}=-\int_{\Omega}f_{j}J(f_{1},\cdots,f_{j-1},\eta,f_{j+1},\cdots,f_{n})d\mathcal{H}^{n}
  \end{equation}
  holds for each $1\leq j\leq n.$
\end{lem}
\begin{thm}\label{thm20}
Suppose that $f$ belongs to  $N_{loc}^{1,n}(\mathcal{S},\mathbb{R}^{n})$ and satisfies the following inequality
\begin{equation}
\| {\rm ap}Df(x)\|^{n}\leq K(x)J_{f}(x)+\Sigma(x)
\end{equation}
for almost everywhere $x\in\mathcal{S}.$ If $K\in L_{loc}^{p}(\mathcal{S})$ for some $p>n$ and  $\Sigma\in L_{loc}^{1+\varepsilon}(\mathcal{S})$ for some $\varepsilon>0,$ then $f$ is locally H\"{o}lder continuous.
\end{thm}
\begin{proof}
By considering the mapping $g:=(f_{1}-(f_{1})_{2B},f_{2},\cdots,f_{n}),$ we get according to Lemma \ref{Kir16} that 
\begin{equation}\label{e2wegystjst}
\begin{aligned}
&\abs{ \int_{2B }\eta J_{f}d\mathcal{H}^{n}}\leq \frac{c(n)}{r}\int_{ 2B}\abs{f_{1}-(f_{1})_{2B}}\left(\frac{\| {\rm ap}Df(x)\|^{n}}{K}\right)^{\frac{n-1}{n}}K^{\frac{n-1}{n}}d\mathcal{H}^{n}\\
\leq&\frac{c(n)}{r} \left(\int_{ 2B}\abs{f_{1}-(f_{1})_{2B}}^{n^{2}}d\mathcal{H}^{n}\right)^{\frac{1}{n^{2}}}\left(\int_{ 2B}K^{p}d\mathcal{H}^{n}\right)^{\frac{n-1}{np}}
\\
&\;\;\;\;\;\;\;\;\;\;\;\;\;\;\;\;\;\;\;\;\;\;\;\times\left(\int_{ 2B}\left(\frac{\| {\rm ap}Df(x)\|^{n}}{K}\right)^{\frac{np}{(p-1)n+p}}d\mathcal{H}^{n}\right)^{\frac{(n-1)[(p-1)n+p]}{n^{2}p}}\\
\leq& \frac{c(n)}{r}\left(\int_{ 2\lambda B}\| {\rm ap}Df(x)\|^{\frac{n^{2}}{n+1}}d\mathcal{H}^{n}\right)^{\frac{n+1}{n^{2}}}\left(\int_{ 2B}K^{p}d\mathcal{H}^{n}\right)^{\frac{n-1}{np}}
\\
&\;\;\;\;\;\;\;\;\;\;\;\;\;\;\;\;\;\;\;\;\;\;\;\times\left(\int_{ 2B}\left(\frac{\| {\rm ap}Df(x)\|^{n}}{K}\right)^{\frac{np}{(p-1)n+p}}d\mathcal{H}^{n}\right)^{\frac{(n-1)[(p-1)n+p]}{n^{2}p}}\\
\leq& \frac{c(n)}{r}\left(\int_{ 2\lambda B}\left(\frac{\| {\rm ap}Df(x)\|^{n}}{K}\right)^{\frac{np}{(p-1)n+p}}d\mathcal{H}^{n}\right)^{\frac{(p-1)n+p}{n^{2}p}}\left(\int_{ 2\lambda B}K^{p}d\mathcal{H}^{n}\right)^{\frac{1}{np}}
\\
\times&\left(\int_{ 2B}\left(\frac{\| {\rm ap}Df(x)\|^{n}}{K}\right)^{\frac{np}{(p-1)n+p}}d\mathcal{H}^{n}\right)^{\frac{(n-1)[(p-1)n+p]}{n^{2}p}}\left(\int_{ 2B}K^{p}d\mathcal{H}^{n}\right)^{\frac{n-1}{np}}\\
\leq& \frac{c(n)}{r}\left(\int_{ 2\lambda B}\left(\frac{\| {\rm ap}Df(x)\|^{n}}{K}\right)^{\frac{np}{(p-1)n+p}}d\mathcal{H}^{n}\right)^{\frac{(p-1)n+p}{np}}\left(\int_{ 2\lambda B}K^{p}d\mathcal{H}^{n}\right)^{\frac{1}{p}}.
\end{aligned}
\end{equation}
Hence, we get
\begin{equation}\label{e2wegystjscscdushydicokpl}
\begin{aligned}
&\int_{B}\frac{\| {\rm ap}Df(x)\|^{n}}{K}d\mathcal{H}^{n}\leq\int_{ 2B}\eta\frac{\| {\rm ap}Df(x)\|^{n}}{K}d\mathcal{H}^{n}\leq \int_{2B }\eta J_{f}d\mathcal{H}^{n}+\int_{2B }\eta \frac{\Sigma}{K} d\mathcal{H}^{n}\\
\leq& \frac{c(n)}{r}\left(\int_{ 2\lambda B}\left(\frac{\| {\rm ap}Df(x)\|^{n}}{K}\right)^{\frac{np}{(p-1)n+p}}d\mathcal{H}^{n}\right)^{\frac{(p-1)n+p}{np}}\left(\int_{ 2\lambda B}K^{p}d\mathcal{H}^{n}\right)^{\frac{1}{p}}+\int_{2B }\eta \Sigma d\mathcal{H}^{n}\\
\leq& c(n)\left(\int_{ 2\lambda B}\left(\frac{\| {\rm ap}Df(x)\|^{n}}{K}\right)^{\frac{np}{(p-1)n+p}}d\mathcal{H}^{n}\right)^{\frac{(p-1)n+p}{np}}\left(\int_{ 2\lambda B}K^{p}d\mathcal{H}^{n}\right)^{\frac{1}{p}}+\int_{2B }\eta \Sigma d\mathcal{H}^{n},
\end{aligned}
\end{equation}
which implies that 
\begin{equation}\label{e688888888fghjk}
\begin{aligned}
\int_{ B}\!\!\!\!\!\!\!\!\!\!\; {}-{} \,\,\,\, \frac{\| {\rm ap}Df(x)\|^{n}}{K}d\mathcal{H}^{n}\leq& C(n,\lambda,\mathcal{K})\cdot\left(\int_{ 2\lambda B}\!\!\!\!\!\!\!\!\!\!\!\!\!\!\!\; {}-{} \,\,\,\left(\frac{\| {\rm ap}Df(x)\|^{n}}{K}\right)^{\frac{np}{(p-1)n+p}}d\mathcal{H}^{n}\right)^{\frac{(p-1)n+p}{np}}\\
&\,\,\,\,\,\,\,\,\,\,\,\,\,\,\,\,\,\,\,\,\,\,\,\,\,\,\,\,\,\,\,\,\,\,\,\,\,\,\,\,\,\,\,
+C(n,\lambda)\cdot \int_{ 2\lambda B}\!\!\!\!\!\!\!\!\!\!\!\!\!\!\!\!\; {}-{} \,\,\,\;\;\Sigma \;d\mathcal{H}^{n}.\\
\end{aligned}
\end{equation}
If we donate $q=[(p-1)n+p]/np, r_{0}=(1+\varepsilon)q>q, g=(\| {\rm ap}Df\|^{n}/K)^{1/q}, f=\Sigma^{1/q},$ then the inequality (\ref{e688888888fghjk}) gives us that
$$\int_{ B}\!\!\!\!\!\!\!\!\!\!\!\; {}-{} \,\,\, g^{q}\leq C(n,\lambda,\mathcal{K})\cdot\left[\left(\int_{2\lambda B}\!\!\!\!\!\!\!\!\!\!\!\!\!\!\!\!\; {}-{} \,\,\,\,\,\,g\right)^{q}
+ \int_{ 2\lambda B}\!\!\!\!\!\!\!\!\!\!\!\!\!\!\!\!\; {}-{} \,\,\,\,\;f^{q}\right].$$
Since $p>n,$ we see that $q>1.$ As $f^{r_{0}}=\Sigma^{1+\varepsilon}\in L_{loc}^{1}(\Omega),$ then by a Gehring's lemma on metric measure space \cite[Theorem 3.3]{Gol05}, we see that   there exists $\beta_{1}>1,$ such that $\| {\rm ap}Df\|^{n}/K\in L_{loc}^{\beta_{1}}(\Omega).$ Without loss of generality, we assume that $\| {\rm ap}Df\|^{n}/K\in L^{\beta_{1}}(\Omega).$ We chose the constants $p_{1},q_{1}>0$ such that 
$p=\frac{\beta_{1}p_{1}}{q_{1}}>n-1,1<q_{1}<\beta_{1}$ and $1/p_{1}+1/q_{1}=1.$ Those constants do exist. It is easy to see that they meet the inequalities 
$$p_{1}>\max\{\frac{\beta_{1}+n-1}{\beta_{1}},\frac{\beta_{1}}{\beta_{1}-1}\}\,\,\,{\rm and }\,\,\,1<q_{1}<\beta_{1}.$$
Then, for every compact set $\Omega_{1}\subset\Omega,$
we get by using of H\"{o}lder inequality that
$$\int_{\Omega_{1}}(\| {\rm ap}Df\|^{n})^{\frac{\beta_{1}}{q_{1}}}d\mathcal{H}^{n}
\leq\left(\int_{\Omega_{1}}\left(\frac{\| {\rm ap}Df\|^{n}}{K}\right)^{\beta_{1}}d\mathcal{H}^{n}\right)^{\frac{1}{q_{1}}}\left(\int_{\Omega_{1}}K^{p}d\mathcal{H}^{n}\right)^{\frac{1}{p_{1}}}<\infty.$$
This shows that $f\in N_{loc}^{1,\gamma}(\mathcal{S},\mathbb{R}^{n})$ with $\gamma=n\beta_{1}/q_{1}>n.$ By \cite[Theorem 5.1]{HK00}, we deduce that $f$ is locally H\"{o}lder continuous.
This finishes the proof.
\end{proof}

\subsection{Corollary I: H\"{o}lder regularity for values of finite distortion}
According to  Theorem \ref{thm20}, we have the following corollary on H\"{o}lder regularity for values of finite distortion.
\begin{prop}\label{dahdujkj}
Suppose that $f\in N_{loc}^{1,n}(\mathcal{S},\mathbb{R}^{n})$  satisfies the following inequality
\begin{equation}
\| {\rm ap}Df(x)\|^{n}\leq K(x)J_{f}(x)+\Sigma(x)\abs{f(x)-y_{0}}^{n}
\end{equation}
for almost everywhere $x\in\mathcal{S}$ with $y_{0}\in\mathbb{R}^{n}.$  If $K\in L_{loc}^{p}(\mathcal{S})$ for some $p>n$ and  $\Sigma\in L_{loc}^{1+\varepsilon}(\mathcal{S})$ for some $\varepsilon>0,$ then $f$ is locally H\"{o}lder continuous.
\end{prop}
\begin{proof}
We only need to prove that there exists a constant $Q\in(1,\infty),$ such that  $\Sigma(x)\abs{f(x)-y_{0}}^{n}\in L_{loc}^{Q}(\mathcal{S}).$ Without loss of generality, we may assume that ${\rm diam}\mathcal{S}<\infty,$ $f\in N^{1,n}(\mathcal{S},\mathbb{R}^{n}),$ $K\in L^{p}(\mathcal{S})$ and  $\Sigma\in L^{1+\varepsilon}(\mathcal{S})$ for some $\varepsilon>0.$ Then, we take a constant $Q$ which satisfying $1<Q<1+\varepsilon.$
Thus, for every compact subset $M\subset \mathcal{S},$ we have by H\"{o}lder inequality
\begin{equation}
\begin{aligned}
\int_{M}\left(\Sigma\abs{f-y_{0}}^{n}\right)^{Q}d\mathcal{H}^{n}=&\int_{M}(\Sigma^{1+\varepsilon})^{\frac{Q}{1+\varepsilon}}\abs{f-y_{0}}^{nQ}d\mathcal{H}^{n}\\
\aleq&\left(\int_{M}\Sigma^{1+\varepsilon}d\mathcal{H}^{n}\right)^{\frac{Q}{1+\varepsilon}}\left(\int_{M}\abs{f-y_{0}}^{m}d\mathcal{H}^{n}\right)^{\frac{1+\varepsilon-Q}{1+\varepsilon}}\\
\aleq&\|\Sigma\|_{L^{1+\varepsilon}(\mathcal{S})}^{Q}\left(\int_{M}\abs{f-y_{0}}^{m}d\mathcal{H}^{n}\right)^{\frac{1+\varepsilon-Q}{1+\varepsilon}},
\end{aligned}
\end{equation}
where $m=nQ(1+\varepsilon)/(1+\varepsilon-Q)>n.$ If we let $f=(f_{1},f_{2},\ldots,f_{n}),$ then we have that $f_{j}\in N^{1,n}(\mathcal{S})$ for all $j\in\{1,2,\ldots,n\}.$ Hence, by using of the Truinger inequality
\cite[Theorem 6.1]{HK00} for Newtonian-Sobolev mappings we get that $f_{j}\in L_{loc}^{l}(\mathcal{S})$ for all $j\in\{1,2,\ldots,n\}$ and all $l\geq1.$ In addition, since
$$
\abs{f-y_{0}}^{m}\aleq\abs{f}^{m}+\abs{y_{0}}^{m}\aleq \abs{y_{0}}^{m}+\sum_{i=1}^{n}f_{i}^{m},
$$
we get 
\begin{equation}
\begin{aligned}
\left(\int_{M}\abs{f-y_{0}}^{m}d\mathcal{H}^{n}\right)^{\frac{1+\varepsilon-Q}{1+\varepsilon}}\aleq&\left(\abs{y_{0}}^{m}\mathcal{H}^{n}(M)+
\sum_{i=1}^{n}\int_{M}f_{i}^{m}d\mathcal{H}^{n}\right)^{\frac{1+\varepsilon-Q}{1+\varepsilon}}\\
\aleq&C(y_{0},{\rm diam}\mathcal{S},\varepsilon,Q,\|f\|_{L^{m}(\mathcal{S})})<\infty.
\end{aligned}
\end{equation}
Therefore, according to Theorem \ref{thm20}, we see that  $f$ is locally H\"{o}lder continuous. This completes the proof of Proposition \ref{dahdujkj}.
\end{proof}

\subsection{Corollary II:  Lusin's condition (N) for values of finite distortion}
 \subsubsection{Lusin's condition (N) for values of finite distortion}
A continuous mapping $f:(X,\mu_{1})\rightarrow(Y,\mu_{2})$ is said to satisfy Lusin's condition (N) if  for each $E\subset X$ with $\mu_{1}(E)=0,$ we have $\mu_{2}(f(E))=0.$ 
Suppose that $f\in N^{1,p}(\Omega,\mathbb{R}^{n})\cap \mathcal{C}(\Omega,\mathbb{R}^{n}),$ where $\Omega\subset\mathcal{S}$ is a domain. In the case that $p>n,$ then by \cite[Theorem 6.1]{HKST01} and \cite[Theorem 1]{Zap14}, we see that $f$ satisfies Lusin's condition (N). In the case that $p<n,$ the mapping $f$ fails to satisfy Lusin's condition (N); see for example in \cite{Pon71}. In the case that 
$p=n,$ there are lots of conditions posed on $f$ which ensure that it satisfies Lusin's condition (N). For example, Zapadinskaya \cite[Theorem 1]{Zap14} showed  that if $f$ is H\"{o}lder continuous, then it satisfies Lusin's condition (N). Therefore, together with Theorem \ref{thm20} and Proposition \ref{dahdujkj}, these yield the following results on Lusin's condition (N).
\begin{prop}\label{porhdjxja}
Suppose that $f$ belongs to  $N_{loc}^{1,n}(\mathcal{S},\mathbb{R}^{n})$ and satisfies the following inequality
\begin{equation}
\| {\rm ap}Df(x)\|^{n}\leq K(x)J_{f}(x)+\Sigma(x)
\end{equation}
for almost everywhere $x\in\mathcal{S}.$ If $K\in L_{loc}^{p}(\mathcal{S})$ for some $p>n$ and  $\Sigma\in L_{loc}^{1+\varepsilon}(\mathcal{S})$ for some $\varepsilon>0,$ then $f$ satisfies locally Lusin's 
condition (N).
\end{prop}
\begin{prop}\label{dahdujkjdfvs} 
Suppose that $f\in N_{loc}^{1,n}(\mathcal{S},\mathbb{R}^{n})$  satisfies the following inequality
\begin{equation}
\| {\rm ap}Df(x)\|^{n}\leq K(x)J_{f}(x)+\Sigma(x)\abs{f(x)-y_{0}}^{n}
\end{equation}
for almost everywhere $x\in\mathcal{S}$ with $y_{0}\in\mathbb{R}^{n}.$ If $K\in L_{loc}^{p}(\mathcal{S})$ for some $p>n$ and  $\Sigma\in L_{loc}^{1+\varepsilon}(\mathcal{S})$ for some $\varepsilon>0,$ then $f$ satisfies locally Lusin's 
condition (N).
\end{prop}
Here we provide an alternative proof of Propositions \ref{porhdjxja} or \ref{dahdujkjdfvs}, which may be of independent interest. We address this in the subsection \ref{ajdbaxb}.
\subsubsection{Pseudomonotone mappings}\label{ajdbaxb}
We first recall the definition of  a mapping to be $K$-pseudomonotone in Euclidean space. Suppose that $ \Omega\subset\mathbb{R}^{n}$ is a domain. In \cite{MM95},  Mal\'{y} and Martio \cite{MM95} called a mapping $f:\Omega\rightarrow \mathbb{R}^{n}$ 
is $K$-pseudomonotone, $K\geq1,$ if
\begin{equation}\label{hgvbnamn}
{\rm diam}(fB(x,r))\leq K{\rm diam}(f\partial B(x,r))
\end{equation}
whenever $B(x,r)$ is compactly contained on $\Omega.$ Then, they showed that a continuous mapping $f\in W^{1,n}(\Omega,\mathbb{R}^{n})$ satisfies Lusin's condition (N) if it is $K$-pseudomonotone.
This result was extended to the setting of metric measure spaces by Heinonen et al. \cite[Theorem 7.2]{HKST01}. In addition,  by a result of Zapadinskaya \cite[Theorem 1]{Zap14} in the setting of metric space, we see that a H\"{o}lder mapping $f:X\rightarrow \mathbb{R}^{n}$ also satisfies Lusin's condition (N). Actually, the condition of H\"{o}lder continuity gives us that 
\begin{equation}\label{xabkkauia}
{\rm diam}(fB(x,r))\leq Kr^{\alpha},0<\alpha<1,
\end{equation}
 whenever $B(x,r)$ is compactly contained on $\Omega.$ If we combing the conditions (\ref{hgvbnamn}) with (\ref{xabkkauia}) in the setting of Euclidean space,   one can also deduce the Lusin's condition (N)
 for $f:\Omega\rightarrow \mathbb{R}^{n}.$ Indeed, this is a result due to Putten \cite[Theorem 3.1]{Put03}, which asserts that if $f\in W^{1,n}(\Omega,\mathbb{R}^{n})\cap \mathcal{C}(\Omega,\mathbb{R}^{n})$  satisfies the following inequality 
 \begin{equation}\label{hahdjad}
 {\rm diam}(fB(x,r))\leq K{\rm diam}(f\partial B(x,r))+K'r^{\alpha}
 \end{equation} 
for every $B(x,r)\subset\subset\Omega.$ Here, $K,K'$ are positive constants depend only on $n$ and $\Omega$ and $0<\alpha\leq1.$ Then $f$ satisfies the Lusin's condition (N).

We next provide a version of this result in the setting of metric measure spaces. We recall that a function $P:[0,\infty]\rightarrow[0,\infty]$ is to be a Young function if it is convex, non-constant in $(0,\infty)$ with $P(0)=0.$ Then, by compared with (\ref{hahdjad}), we introduce the following concept.
\begin{defn}\label{019001sfgvbsdvfvg b}
 Suppose that $\Omega$ is a domain of $\mathcal{S}$ and $P$ is a Young function. Then, we say a mapping $\phi:\Omega\rightarrow \mathbb{R}^{n}$ is $(K,P)$-pseudomonotone on $\Omega$ if for each $x\in\Omega,$ we have the following inequality
\begin{equation}\label{jhxjhjca}
{\rm diam}(fB(x,r))\leq K{\rm diam}(f\partial B(x,r))+K'\sqrt[n]{P(r)}
\end{equation}
 whenever $B(x,r)$ is compactly contained on $\Omega.$ Here, $K$ and $K'$ are two positive constants depend only on $n$ and $\Omega.$
  \end{defn}
The main result of this subsection is the following theorem.
\begin{thm}\label{axkjlax}
  Suppose that a continuous mapping $f\in N^{1,n}(\Omega,\mathbb{R}^{n})$  is $(K,P)$-pseudomonotone on domain  $\Omega\subset\mathcal{S}$ with $P$ satisfying that $P(2t)\aleq P(t)$ as $t\rightarrow0^{+}.$ Then $f$ satisfies the Lusin's condition (N).
\end{thm}
\begin{rem}
Note that taking $P(r)=r^{\alpha}$ $,0<\alpha\leq n,$ in (\ref{jhxjhjca}) yields condition (\ref{hahdjad}). We also observe that Theorem \ref{axkjlax} is new even in Euclidean space. Indeed, there exists a Young function $P$ which satisfying $r^{\beta}\aleq P(r)$ and $P(2t)\aleq P(t)$ for small $r>0$ and any $\beta>0.$ In addition, the proof of \cite[Theorem 3.1]{Put03} relies on the trace theorem for Sobolev mappings in Euclidean space. In contrast, the proof of Theorem \ref{axkjlax} below does not require this result. Here, it also need to note that we use the $P$-Hausdroff measure of $E$ defined by
$$\mathcal{H}^{P}(E):=\sup_{\delta>0}\inf\sum_{i}P(r_{i}),$$
where the infimum is taken over all the coverings of the set $E$ by balls $B_{i}$ with radius $r_{i}$ not exceeding $\delta.$
\end{rem}
\begin{proof}
The proof of this result is similar to \cite[Theorem 7.2]{HKST01}. We must show that $\mathcal{H}^{n}(f(U))=0$ for every $U\subset\subset\Omega$ with $\mathcal{H}^{P}(U)=0.$ Without loss of generality, we may assume that 
$U\subset B(x_{1}, R_{1})$ and $g$ is an upper gradient of $f$ in some neighborhood of $U.$  For every fix $\varepsilon>0,$ we additionally assume that $\Omega\supset U$ is a fix open domain with which satisfying 
 $$\int_{\Omega}g^{n}d\mathcal{H}^{n}<\varepsilon.$$
We consider the set 
$$U_{1}=\left\{x\in U: \int_{B(x,10\lambda r)}g^{n}d\mathcal{H}^{n}<2L\int_{B(x,r/5)}g^{n}d\mathcal{H}^{n}\right\}$$
and denote $U_{2}=U\setminus U_{1}.$ Then, we must show that $\mathcal{H}^{n}(f(U_{1}))=\mathcal{H}^{n}(f(U_{2}))=0.$

We first show that $\mathcal{H}^{n}(f(U_{1}))=0.$ For every $x\in U_{1},$ then we chose a $r_{x}>0$ such that $B(x,10 \lambda r_{x})\subset\Omega$ and 
$$\int_{B(x,10\lambda r_{x})}g^{n}d\mathcal{H}^{n}<2L\int_{B(x,r_{x}/5)}g^{n}d\mathcal{H}^{n}.$$
Then, by \cite[Theorem 7.1]{HK00} we see that  there exists a radius $r\in(r_{x},2r_{x}),$ such that 
\begin{equation}\label{dsfvgbd}
\abs{f(x_{0})-f(y_{0})}\leq c\left(\int_{B(x,10\lambda r_{x})}g^{n} \;d\mathcal{H}^{n}\right)^{1/n}
\end{equation}
for every $x_{0},y_{0}\in\Omega$ satisfying $d(x_{0},x)=d(y_{0},x)=r.$ Now, by using of the $(K,P)$-pseudomonotone of $f,$ we obtain
\begin{equation}\label{1266dcsgbd}
\begin{aligned}
({\rm diam}f(B(x,r_{x})))^{n}\leq& M_{1}({\rm diam}f(\partial B(x,r)))^{n}+M_{2}P(r)\\
\leq& M_{1}\int_{B(x,10\lambda r_{x})}g^{n} \;d\mathcal{H}^{n} +M_{2}P(2r_{x})\\
\leq& M_{1}\int_{B(x,10\lambda r_{x})}g^{n} \;d\mathcal{H}^{n} +M_{2}P(r_{x}).\\
\end{aligned}
\end{equation}
Hence, we get that $U_{1}\subset\cup_{x\in U_{1}}B(x,r_{x})$ with $B(x,r_{x})$ satisfying the inequality (\ref{1266dcsgbd}).  By covering theorem \cite[Chapter 1]{Hei01}, we can chose a countable collection
$\{B(x_{i},r_{i})\}$ which covers $U_{1}$ and $\{B(x_{i},r_{i}/5)\}$ is pairwise disjoint.
By using this covering and inequality (\ref{1266dcsgbd}), we obtain
\begin{equation}\label{12sdvv vd66dcsgbd}
\begin{aligned}
&\sum_{i}({\rm diam}f(B(x_{i},r_{i})))^{n}\\
\leq& M_{1}\sum_{i}\int_{B(x_{i},10\lambda r_{i})}g^{n} \;d\mathcal{H}^{n}+M_{2}\sum_{i}P(r_{i})\\
<& M_{1}\sum_{i}\int_{B(x_{i}, r_{i}/5)}g^{n} \;d\mathcal{H}^{n} +M_{2}\sum_{i}P(r_{i})\\
\leq& M_{1}\int_{\Omega}g^{n} \;d\mathcal{H}^{n} +M_{2}\sum_{i}P(r_{i})<(M_{1}+M_{2})\varepsilon,\\
\end{aligned}
\end{equation}
which implies that $\mathcal{H}^{n}(f(U_{1}))=0$ by letting $\varepsilon\rightarrow0^{+}.$

Next, we show that $\mathcal{H}^{n}(f(U_{2}))=0.$ This can be done, according to \cite[Corollary 3.4 and Proposition 4.1.]{GO18}, by showing that for every 
$a\in U_{2},$ there has $${\rm Lip}(a):=\limsup_{x\rightarrow a}\frac{\vert{f(x)-f(a)\vert}}{\vert{x-a\vert}}<\infty.$$
Since $a\in U_{2},$ we see that there exists an integer $m>1,$ such that for every $0<r<1/m,$ there holds the following 
\begin{equation}\label{12sd}
2L\int_{B(a,r/5)}g^{n}d\mathcal{H}^{n}\leq\int_{B(a,10\lambda r)}g^{n}d\mathcal{H}^{n}.
\end{equation}\label{dcfevf12sd}
If we let $\omega(r)=\int_{B(a,r/5)}g^{n}d\mathcal{H}^{n},$ then the inequality (\ref{12sd}) gives us that 
\begin{equation}
\omega(r)\leq \omega(50\lambda r)/(2L).
\end{equation} For every fix $0<t<1/m,$ we chose 
$k\in \mathbb{N}$ with $(50\lambda)^{-k-1}\leq mr<(50\lambda)^{-k}.$ Then, by the inequality (\ref{dcfevf12sd}), we get for every $0<r<1/m$
$$\omega(r)\leq Cr^{n}.$$ Now, we take every $x\in U_{2}$ with $x\in B(a,1/10m),$ according to \cite[Lemma 4.6]{HKST01}, we deduce that 
$$\limsup_{x\rightarrow a}\frac{\vert f(x)-f(a)\vert}{\vert x-a\vert}\leq C\limsup_{x\rightarrow a}(M_{1/m,n}g(x)+M_{1/m,n}g(a))\leq C.$$
This finishes the proof.
\end{proof}

\section{Totally disconnected  for values of finite distortion and its corollaries}\label{sec3333}

\subsection{The totally disconnected}
In this section, we will study  the totally disconnected  for values of finite distortion from generalized $n$-manifold with controlled geometry to Euclidean space.
Firstly, we need the following Caccioppoli-type estimate deduced from Lemma \ref{Kir16}. Estimates of this kind  in Euclidean space have been established by  Kangasniemi and  Onninen \cite[Lemma 6.2]{KO221} and  Onninen and Zhong \cite[Lemma 2.1]{OZ08}. 
\begin{lem}\label{shksujck}
Suppose that $f\in N_{loc}^{1,n}(\Omega,\mathbb{R}^{n})\cap L_{loc}^{\infty}(\Omega,\mathbb{R}^{n}),$ $\Omega\subset\mathcal{S}$ is a relatively compact open set and $\eta:\Omega\rightarrow \mathbb{R}$ is a Lipschitz mapping with 
  ${\rm spt}\eta\subset\Omega.$ If $\varphi:[0,\infty)\rightarrow [0,\infty)$ is piecewise $C^{1}$-smooth function with $\varphi'$ is locally bounded, then we have
  \begin{equation}
  \abs{\int_{\Omega}\eta [n\varphi(\abs{f}^{2})+2\varphi'(\abs{f}^{2})\abs{f}^{2}]J_{f}}
  \aleq\int_{\Omega}\varphi(\abs{f}^{2})\abs{f}\abs{\nabla\eta}\| {\rm ap}Df\|^{n-1}.
  \end{equation}
\end{lem}
\begin{proof}
Without loss of generality, we may assume that $\Omega$ is an open and bounded domain such that $f\in N^{1,n}(\Omega,\mathbb{R}^{n})\cap L^{\infty}(\Omega,\mathbb{R}^{n}),$ and $\varphi$ and $\varphi'$
are bounded. If by letting that $F_{i}=(f_{1},f_{2},\ldots,\varphi(\vert f\vert^{2})f_{i},\ldots,f_{n}),$ then we have  $F_{i}\in N^{1,n}(\Omega,\mathbb{R}^{n})\cap L^{\infty}(\Omega,\mathbb{R}^{n}).$ To show this we must verify that 
$\varphi(\vert f\vert^{2})f_{i}\in N^{1,n}(\Omega)\cap L^{\infty}(\Omega).$ Since  $f_{j}\in N^{1,n}(\Omega)\cap L^{\infty}(\Omega)$ for all $1\leq j\leq n,$ and 
$\varphi$ and $\varphi'$ are bounded,  we only verify that $\vert D(\varphi(\vert f\vert^{2})f_{j})\vert\in L_{loc}^{n}(\Omega).$   This conclusion follows by 
$$\vert D(\varphi(\vert f\vert^{2})f_{i})\vert\leq2\vert\varphi'(\vert f\vert^{2})\vert\vert f_{i}\vert\sum_{j=1}^{n}\vert f_{i}\vert\vert Df_{j}\vert+\vert\varphi(\vert f\vert^{2})\vert\vert Df_{i}\vert<\infty.$$  
Now, by virtue of Lemma \ref{Kir16}, we get 
\begin{equation}
\begin{aligned}
&\abs{\int_{\Omega}\eta [\varphi(\abs{f}^{2})+2\varphi'(\abs{f}^{2})\vert f\vert^{2}] J_{f}d\mathcal{H}^{n}}\\
\aleq&\sum_{j=1}^{n}\int_{\Omega}\varphi(\abs{f}^{2}) \vert f_{j}\vert\vert J(f_{1},\cdots,f_{j-1},\eta,f_{j+1},\cdots,f_{n})\vert d\mathcal{H}^{n}\\
\aleq&\int_{\Omega}\varphi(\abs{f}^{2}) \vert f\vert\abs{\nabla\eta}\|{\rm ap}Df\|^{n-1} d\mathcal{H}^{n}.
\end{aligned}
\end{equation}
This finishes the proof.
\end{proof}
We also need the following basic fact. 
\begin{lem}\label{shksujcfdihikd}
Suppose that  $x\geq0, n\geq2$ and $M,N$ are two positive constants which are independent of $x.$  If  $x\leq Mx^{\frac{n-1}{n}}+N,$ then there exists a constant $c=c(M,N,n)>0,$ such that $x\leq c(M,N,n).$
\end{lem}
\begin{proof}
If $x=0,$ then we can chose any positive constant $c$ such that $x\leq c.$ Now, we assume that $x>0$ and consider the function $$\mathcal{G}(x)=x^{1/n}-\frac{N}{x^{1-\frac{1}{n}}},x>0.$$
It is obvious that $\lim_{x\rightarrow0^{+}}\mathcal{G}(x)=-\infty$ and $\lim_{x\rightarrow+\infty}\mathcal{G}(x)=+\infty$ and $\mathcal{G}$ is strictly monotonically increasing when $x>0.$ Hence, for any $t>0,$ there exists 
a unique $c=c(n,N,t)>0$ such that $\mathcal{G}(x)=t.$ So the desired result follows by letting $t=M$ and $c(M,N,n)=\mathcal{G}^{-1}(M).$
\end{proof}

By virtue of Lemmas \ref{shksujck} and \ref{shksujcfdihikd}, we have the following result. For simplicity, we introduce the following symbol  $$A(p,q,n):=\left\{(p,q,n):p>n-1,n>2,q>1\;{\rm and}\;\frac{1}{p}+\frac{1}{q}<1\right\}.$$
\begin{thm}\label{th11000}
Let $\Omega\subset\mathcal{S}$ be a bounded domain and $y_{0}\in\mathbb{R}^{n}$. Suppose that $f$ belongs to  $N_{loc}^{1,n}(\Omega,\mathbb{R}^{n})\cap \mathcal{C}(\Omega,\mathbb{R}^{n})$ and satisfies the following inequality
\begin{equation}
\| {\rm ap}Df(x)\|^{n}\leq K(x)J_{f}(x)+\Sigma(x)\abs{f(x)-y_{0}}^{n}
\end{equation}
for almost everywhere $x\in\Omega.$ If $K\in L_{loc}^{p}(\Omega)$ and $\Sigma/K\in L_{loc}^{q}(\Omega)$ with $(p,q)\in A(p,q,n),$  then 
$g\in L_{loc}^{q_{1}}(\Omega),$ where 
\begin{equation}\label{xgaufhgjhjkkixkx}
  g(x):=\left\{
\begin{aligned}
&g_{u}   &   &    {\rm if}\;{x\in \Omega\setminus f^{-1}\{y_{0}\}},\\
&0       &      &{\rm if}\; {x\in f^{-1}\{y_{0}\}}
\end{aligned} \right.
\end{equation}
 and $u(x)=\log^{+}\log^{+}1/\vert f(x)-y_{0}\vert$ and $q_{1}=pn/(1+p)\in(n-1, n).$ 
\end{thm}
\begin{proof}
Without loss of generality, we may assume that $y_{0}=0.$ Since $ f$ is continuous on $\Omega$,  the set $$\Omega_{1}:=\{x\in\Omega:\abs{f(x)}<1/e^{2}\}$$ is an open domain and $\log^{+}\log^{+}(1/\vert f(x)\vert) =0$ outside of $\Omega_{1}.$ So we may, without loss of generality, consider that $\Omega_{1}=\Omega.$  We also assume that $f\in N^{1,n}(\Omega,\mathbb{R}^{n})\cap \mathcal{C}(\Omega,\mathbb{R}^{n}),K\in L^{p}(\Omega)$ and and $\Sigma/K\in L^{q}(\Omega).$ If we denote $$u(x)=\log\log\frac{1}{\vert f(x)\vert},x\in\Omega,$$ then by \cite[Proposition 2.17]{BB11}, we see that 
$u\in N^{1,n}(\Omega,\mathbb{R}^{n})\cap \mathcal{C}(\Omega,\mathbb{R}^{n}).$ Now, we fix $\varepsilon>0$ and consider the following mapping 
$$\varphi_{\varepsilon}(t)=\frac{1}{2t^{\frac{n}{2}}}\int_{0}^{t}\frac{\phi_{\varepsilon}(s)}{s\log^{n}(\frac{1}{s})}ds,$$
where $$\phi_{\varepsilon}(s)=\frac{1}{1+\varepsilon2^{\frac{1}{s}}}.$$
For every fixed compact subset $\Omega_{1}\subset\Omega,$ we take mappings $\varphi=\varphi_{\varepsilon}$ and $\eta=\eta_{1}^{n}.$ Here $\eta_{1}$ is a Lipschitz map with ${\rm spt}\eta_{1}\in \Omega$ satisfying $\eta_{1}=1$ on $\Omega_{1},$ $\vert\nabla\eta_{1}\vert\leq1/r$ on $\Omega$ and $\eta_{1}=0$ on $\Omega\setminus\Omega_{1}.$
Then, by  virtue of Lemma \ref{shksujck}, we obtain 
\begin{equation}\label{09}
\begin{aligned}
&\int_{\Omega}\eta_{1}^{n}\frac{J_{f}}{\abs{f}^{n}\log^{n}(\frac{1}{\abs{f}})}\varphi_{\varepsilon}(\abs{f})d\mathcal{H}^{n}\\
\aleq& \int_{\Omega}\eta_{1}^{n-1}\vert\nabla\eta_{1}\vert\frac{\parallel ap Df\parallel^{n-1}}{\vert f\vert^{n-1}\log^{n-1}(\frac{1}{\vert f\vert})}\left(\varphi_{\varepsilon}(\vert f\vert)\right)^{\frac{n-1}{n}}d\mathcal{H}^{n}\\
\aleq&  \left[\int_{\Omega}\eta_{1}^{n}\frac{\parallel ap Df\parallel^{n}}{K\vert f\vert^{n}\log^{n}(\frac{1}{\vert f\vert})}\varphi_{\varepsilon}(\vert  f\vert)d\mathcal{H}^{n}\right]^{\frac{n-1}{n}}\cdot\left[\int_{\Omega}K^{n-1}\vert\nabla\eta_{1}\vert^{n}d\mathcal{H}^{n}\right]^{\frac{1}{n}}.\\
\end{aligned}
\end{equation}
Hence,  by combining $\frac{\parallel ap Df\parallel^{n}}{K}-\frac{\Sigma\vert f\vert^{n}}{K}\leq J_{f}$ and inequality (\ref{09}), we get that 
\begin{equation}\label{07}
\begin{aligned}
&\int_{\Omega}\eta_{1}^{n}\frac{\parallel ap Df\parallel^{n}}{K\vert f\vert^{n}\log^{n}(\frac{1}{\vert f\vert})}\varphi_{\varepsilon}(\vert  f\vert)d\mathcal{H}^{n}\\
\aleq&  \left[\int_{\Omega}\eta_{1}^{n}\frac{\parallel ap Df\parallel^{n}}{K\vert f\vert^{n}\log^{n}(\frac{1}{\vert f\vert})}\varphi_{\varepsilon}(\vert  f\vert)d\mathcal{H}^{n}\right]^{\frac{n-1}{n}}\cdot\left[\int_{\Omega}K^{n-1}\vert\nabla\eta_{1}\vert^{n}d\mathcal{H}^{n}\right]^{\frac{1}{n}}\\
&\;\;\;\;\;\;\;\;\;\;\;\;\;\;\;\;\;\;\;\;\;\;\;\;\;\;\;\;\;\;\;\;\;\;+\int_{\Omega}\eta_{1}^{n}\frac{\varphi_{\varepsilon}(\vert  f\vert)}{\log^{n}(\frac{1}{\vert f\vert})}\frac{\Sigma}{K}d\mathcal{H}^{n}\\
\leq& C(n,\parallel K^{n-1}\vert\nabla\eta_{1}\vert^{n}\parallel_{L_{1}(\Omega)}) \left[\int_{\Omega}\eta_{1}^{n}\frac{\parallel ap Df\parallel^{n}}{K\vert f\vert^{n}\log^{n}(\frac{1}{\vert f\vert})}\varphi_{\varepsilon}(\vert  f\vert)d\mathcal{H}^{n}\right]^{\frac{n-1}{n}}\\
&\;\;\;\;\;\;\;\;\;\;\;\;\;\;\;\;\;\;\;\;\;\;\;\;\;\;\;\;\;\;\;\;\;\;\;\;\;\;\;\;\;\;\;\;\;\;\;\;\;\;\;\;\;\;\;\;\;\;\;+\int_{\Omega}\eta_{1}^{n}\frac{\varphi_{\varepsilon}(\vert  f\vert)}{\log^{n}(\frac{1}{\vert f\vert})}\frac{\Sigma}{K}d\mathcal{H}^{n}.\\
\end{aligned}
\end{equation}
Since $$\int_{\Omega}\eta_{1}^{n}\frac{\varphi_{\varepsilon}(\vert  f\vert)}{\log^{n}(\frac{1}{\vert f\vert})}\frac{\Sigma}{K}d\mathcal{H}^{n}\aleq\int_{\Omega}\frac{\Sigma}{K} d\mathcal{H}^{n}
\aleq\left\|\frac{\Sigma}{K}\right\|_{L^{q}(\Omega)}(\mathcal{H}^{n}(\Omega))^{\frac{q-1}{q}}<\infty,$$
we conclude that, by using of Lemma \ref{shksujcfdihikd} and inequality (\ref{07}), there exists a constant $C= C(n,\parallel K^{n-1}\vert\nabla\eta_{1}\vert^{n}\parallel_{L^{1}(\Omega)},\left\|\Sigma/K\right\|_{L^{q}(\Omega)},\Omega),$ such that 
\begin{equation}\label{sxnlxk}
\int_{\Omega}\eta_{1}^{n}\frac{\| ap Df\|^{n}}{K\vert f\vert^{n}\log^{n}(\frac{1}{\vert f\vert})}\varphi_{\varepsilon}(\vert  f\vert)d\mathcal{H}^{n}\leq C(n,\| K^{n-1}\vert\nabla\eta_{1}\vert^{n}\|_{L^{1}(\Omega)},\left\|\Sigma/K\right\|_{L^{q}(\Omega)},\Omega).
\end{equation}
By dominated convergence theorem, we get that 
\begin{equation}\label{sxnlxk}
\int_{\Omega}\eta_{1}^{n}\frac{\| ap Df\|^{n}}{K\vert f\vert^{n}\log^{n}(\frac{1}{\vert f\vert})}d\mathcal{H}^{n}\leq C(n,\| K^{n-1}\vert\nabla\eta_{1}\vert^{n}\|_{L^{1}(\Omega)},\left\|\Sigma/K\right\|_{L^{q}(\Omega)},\Omega).
\end{equation}
Let $q_{1}=pn/(1+p),$  then we deduce that 
\begin{equation}\label{04sfbsf}
\begin{aligned}
&\int_{\Omega_{1}}\frac{\| ap Df\|^{q_{1}}}{\vert f\vert^{q_{1}}\log^{q_{1}}(\frac{1}{\vert f\vert})}d\mathcal{H}^{n}\\
=&\int_{\Omega_{1}}\left[\frac{\|ap Df\|^{n}}{K\vert f\vert^{n}\log^{n}(\frac{1}{\vert f\vert})}\right]^{\frac{q_{1}}{n}}\cdot K^{\frac{q_{1}}{n}}d\mathcal{H}^{n}\\
\leq& \left[\int_{\Omega_{1}}\frac{\| ap Df\|^{n}}{K\vert f\vert^{n}\log^{n}(\frac{1}{\vert f\vert})}\right]^{\frac{q_{1}}{n}}\cdot \left[\int_{\Omega_{1}}K^{p}(x)dx\right]^{\frac{1}{p}}\\
\leq& \left[\int_{\Omega}\frac{\| ap Df\|^{n}}{K\vert f\vert^{n}\log^{n}(\frac{1}{\vert f\vert})}\right]^{\frac{q_{1}}{n}}\cdot \left[\int_{\Omega}K^{p}(x)dx\right]^{\frac{1}{p}}<\infty.
\end{aligned}
\end{equation}
Therefore, by using of \cite[Proposition 2.17]{BB11} and H\"{o}lder inequality, we get for almost every $x\in\Omega\setminus f^{-1}\{0\}$  that 
\begin{equation}
\begin{aligned}
\int_{\Omega\setminus f^{-1}(0)}g_{u}^{q_{1}} d\mathcal{H}^{n}=&\int_{\Omega\setminus f^{-1}(0)}\frac{\| apDf\|^{q_{1}}}{\vert f\vert^{q_{1}}\log^{q_{1}}(\frac{1}{\vert f\vert})}d\mathcal{H}^{n}\\
\leq& \left[\int_{\Omega}\frac{\| apDf\|^{n}}{K\vert f\vert^{n}\log^{n}(\frac{1}{\vert f\vert})}\right]^{\frac{q_{1}}{Q}}\cdot \left[\int_{\Omega}K^{p}(x)dx\right]^{\frac{1}{p}}<\infty.\\
\end{aligned}
\end{equation}
This shows that $g\in L_{loc}^{q_{1}}(\Omega).$  This finishes the proof of Theorem \ref{th11000}.
\end{proof}

In order to imply that $f^{-1}\{y_{0}\}$ is totally disconnected, one can show that $\mathcal{H}^{1}(f^{-1}\{y_{0}\})=0.$ The following lemma gives us a way to deduce this conclusion. 

\begin{lem}{\rm\cite[Lemma 7.1]{Kir14}}\label{101209sdsdcsdcc}
Suppose that $\mathcal{S}$ support a weak $(n-1)$-Poincar\'{e} inequality and $1<n-1<p<n.$ Let $E\subset\mathcal{S}$ be a closed set such that there is a sequence $\{u_{i}\}_{i=1}^{\infty}$ of continuous functions with $n$-weak upper gradients $\{g_{i}\}_{i=1}^{\infty}$ respectively, satisfying the following conditions.
\begin{enumerate}
\item[{\rm(1)}] $u_{i}=1$ for every $i\in\mathbb{N}$ and for every $x\in E.$
\item[{\rm(2)}] $\int_{\Omega}(g_{i}(x))^{p}d\mathcal{H}^{n}\rightarrow0$ as $i\rightarrow\infty$ for every compact set $\Omega\subset\mathcal{S}.$
\item[{\rm(3)}] There exists $y\in\mathcal{S}$ and $R>0$ such that $u_{i}(x)\leq1/2$ for every $x\in\mathcal{S}\cap B(y,R)$ and for every $i\in\mathbb{N}.$
\end{enumerate} 
Then $\mathcal{H}_{\infty}^{1}(E)=0$ and thus also $\mathcal{H}^{1}(E)=0.$
\end{lem}

Now, by Theorem \ref{th11000} and Lemmas \ref{101209sdsdcsdcc}, we have the following result.
\begin{thm}\label{101209sdc}
Let  $\Omega\subset\mathcal{S}$ be a domain and $y_{0}\in\mathbb{R}^{n}$. Suppose that an nonconstant mapping $f$ belongs to  $N_{loc}^{1,n}(\Omega,\mathbb{R}^{n})\cap \mathcal{C}(\Omega,\mathbb{R}^{n})$ and satisfies the following inequality
\begin{equation}
\| {\rm ap}Df(x)\|^{n}\leq K(x)J_{f}(x)+\Sigma(x)\abs{f(x)-y_{0}}^{n}
\end{equation}
for almost everywhere $x\in\Omega.$ If $K\in L_{loc}^{p}(\Omega)$ and $\Sigma/K\in L_{loc}^{q}(\Omega)$ with $(p,q)\in A(p,q,n),$  then $f^{-1}\{y_{0}\}$ is totally disconnected.
\end{thm}
\begin{proof}
We also only assume that the case that $y_{0}=0$. For every $i\in\mathbb{N},$ we consider the sequence of continuous mappings $$g_{i}(x)=\frac{\min\{g(x),i\}}{i},$$ 
where $g$ is defined by (\ref{xgauixkx}). Then, $g$ satisfies the following property: 
\begin{enumerate}
\item[{\rm(1)}] $g_{i}\equiv1$ for every $i\in\mathbb{N}$ and for every $x\in f^{-1}\{0\}.$
\end{enumerate} 
As $f$ is an nonconstant continuous mapping, we see that:  
\begin{enumerate}
\item[{\rm(2)}] There exists $y\in\mathcal{S}$ and $R>0$ such that $g_{i}(x)\leq1/2$ for every $x\in\mathcal{S}\cap B(y,R)$ when $i$ is sufficiently large. 
\end{enumerate} 
Then, according to Lemma \ref{101209sdsdcsdcc}, the conclusion of $\mathcal{H}^{1}(f^{-1}(0))=0$ will be deduced if we find an $n$-weak upper gradient $g_{i}$ of $u_{i},$ such that $g_{i}$ satisfying the following assert:
\begin{enumerate}
\item[{\rm(3)}] $\int_{\Omega}(g_{i}(x))^{p}d\mathcal{H}^{n}\rightarrow0$ as $i\rightarrow\infty$ for every compact set $\Omega\subset\mathcal{S}.$
\end{enumerate} 

To do so, we consider $g_{i}(x)=\frac{1}{i}g(x),$ where 
\begin{equation}\label{xgauixkx}
 g(x)=\left\{
\begin{aligned}
&g_{u}   &   &    {\rm if}\;{x\in \mathcal{S}\setminus f^{-1}(0)},\\
&0       &      &{\rm if}\; {x\in f^{-1}(0).}
\end{aligned} \right.
\end{equation}
Since by Theorem \ref{th11000}, we have $g\in L_{loc}^{q_{1}}(\Omega),$ where 
\begin{equation}\label{xgaufhgjhjkkixkx}
  g(x):=\left\{
\begin{aligned}
&g_{u}   &   &    {\rm if}\;{x\in \Omega\setminus f^{-1}(0)},\\
&0       &      &{\rm if}\; {x\in f^{-1}(0)}
\end{aligned} \right.
\end{equation}
 and $u(x)=\log^{+}\log^{+}1/\vert f(x)\vert$ and $q_{1}=pn/(1+p)\in(n-1, n).$ 
Therefore, we get that
$$\int_{\Omega}\vert g_{i}(x)\vert^{q_{1}} d\mathcal{H}^{n} \leq\frac{1}{i}\left[\int_{\Omega}\frac{\|  apDf(x)\|^{n}}{K(x)\vert f(x)\vert^{n}\log^{n}(\frac{1}{\vert f(x)\vert})}d\mathcal{H}^{n}\right]^{\frac{q_{1}}{n}}\cdot \left[\int_{\Omega}K^{p}(x)d\mathcal{H}^{n}\right]^{\frac{1}{p}}\rightarrow0,$$
 as $i\rightarrow\infty.$ This completes the proof.
\end{proof}
 
\subsection{Corollary I: The choice of components} 
Here we list some corollaries on the choice of components of $f^{-1}(\mathbb{B}(y_{0},\varepsilon)).$ Those results were obtained by Kangasniemi and  Onninen \cite{KO222} in the setting of Euclidean space. 
Since these results are topological in nature, we may directly obtain the corresponding conclusions on generalized $n$-manifold with controlled geomerty,  owing to the total disconnectedness established 
in Lemma \ref{101209sdc}. The readers can directly prove Corollaries \ref{9981}, \ref{sjslx} and \ref{sxijsxsx} by following \cite[Lemmas 4.1, 4.2 and 4.3]{KO222}. 
\begin{cor}\label{9981}
Let  $\Omega\subset\mathcal{S}$ be a domain and $y_{0}\in\mathbb{R}^{n}$. Suppose that an nonconstant mapping $f$ belongs to  $N_{loc}^{1,n}(\Omega,\mathbb{R}^{n})\cap \mathcal{C}(\Omega,\mathbb{R}^{n})$ and satisfies the following inequality
\begin{equation}
\| {\rm ap}Df(x)\|^{n}\leq K(x)J_{f}(x)+\Sigma(x)\abs{f(x)-y_{0}}^{n}
\end{equation}
for almost everywhere $x\in\Omega,$ where $K\in L_{loc}^{p}(\Omega)$ for some $p>n-1$ and $\Sigma\in L_{loc}^{1+\varepsilon}(\Omega)$ for some $\varepsilon>0.$ Then, there exists $\varepsilon_{1}>0$ such that if $U_{\varepsilon_{1}}$ is the $x$-component of $f^{-1}(\mathbb{B}(y_{0},\varepsilon_{1})),$ then $\overline{U_{\varepsilon_{1}}}$ is compactly contained in $\Omega.$
\end{cor}

\begin{cor}\label{sjslx}
Let $\Omega\subset\mathcal{S}$ be a domain and $y_{0}\in\mathbb{R}^{n}$. Suppose that an nonconstant mapping  $f$ belongs to  $N_{loc}^{1,n}(\Omega,\mathbb{R}^{n})\cap \mathcal{C}(\Omega,\mathbb{R}^{n})$ and satisfies the following inequality
\begin{equation}
\| {\rm ap}Df(x)\|^{n}\leq K(x)J_{f}(x)+\Sigma(x)\abs{f(x)-y_{0}}^{n}
\end{equation}
for almost everywhere $x\in\Omega,$ where $K\in L_{loc}^{p}(\Omega)$ for some $p>n-1$ and $\Sigma\in L_{loc}^{1+\varepsilon}(\Omega)$ for some $\varepsilon>0.$   Then  all every two different points $x_{1},x_{2}\in f^{-1}\{y_{0}\},$ there exists $\varepsilon_{0}>0,$ such that $x_{1}$ and $x_{2}$ are in different components of $f^{-1}(\mathbb{B}(y_{0},\varepsilon_{0})).$
\end{cor}

\begin{cor}\label{sxijsxsx}
Let $\Omega\subset\mathcal{S}$ be a domain and $y_{0}\in\mathbb{R}^{n}$. Suppose that an nonconstant mapping  $f$ belongs to  $N_{loc}^{1,n}(\Omega,\mathbb{R}^{n})\cap \mathcal{C}(\Omega,\mathbb{R}^{n})$ and satisfies the following inequality
\begin{equation}
\| {\rm ap}Df(x)\|^{n}\leq K(x)J_{f}(x)+\Sigma(x)\abs{f(x)-y_{0}}^{n}
\end{equation}
for almost everywhere $x\in\Omega,$ where $K\in L_{loc}^{p}(\Omega)$ for some $p>n-1$ and $\Sigma\in L_{loc}^{1+\varepsilon}(\Omega)$ for some $\varepsilon>0.$ Let $x_{0}\in f^{-1}\{y_{0}\},$  and for all $\varepsilon>0,$ $U_{\varepsilon}$ be the $x_{0}$-component of $f^{-1}(\mathbb{B}(y_{0},\varepsilon_{0})).$ Then, $\lim_{\varepsilon\rightarrow0^{+}}\mathcal{H}^{n}(U_{\varepsilon})=0.$
\end{cor}

\subsection{Corollary II: A Jacobian-degree formula for values of finite distortion} 
The following Jacobian-degree formula for Newtonian-Sobolev mapping is similar to the one obtained by Kangasniemi and  Onninen \cite{KO222} on Euclidean space. 
\begin{lem}\label{lomlm}
Let  $\Omega\subset\mathcal{S}$ be a domain and $y_{0}\in\mathbb{R}^{n}$. Suppose that an nonconstant mapping $f$ belongs to  $N_{loc}^{1,n}(\Omega,\mathbb{R}^{n})\cap \mathcal{C}(\Omega,\mathbb{R}^{n})$ and has Lusin's condition (N). If $f^{-1}\{y_{0}\}$ is totally disconnected, then there $\varepsilon_{1}>0,$ such that $U_{\varepsilon_{1}}\subset\subset\Omega$ is a connected component of $f^{-1}(\mathbb{B}(y_{0},\varepsilon_{1}))$ and satisfies the following Jacobian-degree formula
\begin{equation}
{\rm deg}(f,U_{\varepsilon_{1}})=\frac{1}{\omega_{n}\varepsilon_{1}^{n}}\int_{U_{\varepsilon_{1}}}J_{f}d\mathcal{H}^{n}.
\end{equation}
\end{lem}
\begin{proof}
According to Corollary \ref{9981}, we can chose a $\varepsilon_{1}>0$  such that $U_{\varepsilon_{1}}\subset\subset\Omega$ is a connected component of $f^{-1}(\mathbb{B}(y_{0},\varepsilon_{1}))$  satisfying that $f\in N^{1,n}(U_{\varepsilon_{1}},\mathbb{R}^{n})\cap \mathcal{C}(U_{\varepsilon_{1}},\mathbb{R}^{n})$ and having the Lusin's condition (N) on $U_{\varepsilon_{1}}.$

We first consider that $f$ is Lipschitz map. Then,  by \cite[Lemma 3.5]{Kir16}, we obtain that
$${\rm deg}(y,f,U_{\varepsilon_{1}})=\sum_{x\in f^{-1}(y)\cap U_{\varepsilon_{1}}}{\rm sgn}J_{f}(x)$$
for $\mathcal{H}^{n}$-almost every $y\in \mathbb{R}^{n}\setminus f(\partial U_{\varepsilon_{1}}).$ Combing this with coarea formula \cite[Theorem 3.3]{Kir16}, we obtain
\begin{equation}\label{shij}
\int_{U_{\varepsilon_{1}}}J_{f}(x)d\mathcal{H}^{n}(x)=\int_{\mathbb{R}^{n}}{\rm deg}(y,f,U_{\varepsilon_{1}})d\mathcal{H}^{n}(y).
\end{equation}
Next, by  \cite[Lemma 4.3 ]{Kir14}, we can select a sequence of locally Lipschitz maps $\{f_{j}\}_{j=1}^{\infty},$ such that $f_{j}\rightarrow f$ uniformly on compact sets. 
 Thus, for sufficiently large indices $j$ there is a homotopy between $f_{j}$ and $f,$ such that ${\rm deg}(y,f_{j},U_{\varepsilon_{1}})={\rm deg}(y,f,U_{\varepsilon_{1}}).$
 In addition, according to \cite[Corollary 2.31]{Kir14}, we see that $$\int_{U_{\varepsilon_{1}}}J_{f_{j}}d\mathcal{H}^{n}\rightarrow \int_{U_{\varepsilon_{1}}}J_{f}d\mathcal{H}^{n}\;\;{\rm as}\;\;j\rightarrow\infty.$$
 So we get 
 \begin{equation}\label{ssvsvshij}
\int_{U_{\varepsilon_{1}}}J_{f}(x)d\mathcal{H}^{n}(x)=\int_{\mathbb{R}^{n}}{\rm deg}(y,f,U_{\varepsilon_{1}})d\mathcal{H}^{n}(y)
\end{equation}
when $f\in N^{1,n}(U_{\varepsilon_{1}},\mathbb{R}^{n})\cap \mathcal{C}(U_{\varepsilon_{1}},\mathbb{R}^{n}).$ Now, from  \cite[Lemma 2.7]{Kir16} we deduce that 
$$\int_{U_{\varepsilon_{1}}}J_{f}(x)d\mathcal{H}^{n}(x)=\int_{\mathbb{B}(y_{0},\varepsilon_{1})}{\rm deg}(y,f,U_{\varepsilon_{1}})d\mathcal{H}^{n}(y)=\omega_{n}\varepsilon_{1}^{n}{\rm deg}(f,U_{\varepsilon_{1}}).$$
This finishes the proof of Lemma \ref{lomlm}.
\end{proof}
According to Proposition \ref{porhdjxja} and Lemma \ref{lomlm}, we can directly deduce the following Jacobian-degree formula for values of finite distortion. 
\begin{lem}\label{lomlmass}
Let  $\Omega\subset\mathcal{S}$ be a domain and $y_{0}\in\mathbb{R}^{n}$.  Suppose that $f$ belongs to  $N_{loc}^{1,n}(\mathcal{S},\mathbb{R}^{n})$ and satisfies the following inequality
\begin{equation}
\| {\rm ap}Df(x)\|^{n}\leq K(x)J_{f}(x)+\Sigma(x)\abs{f(x)-y_{0}}^{n}
\end{equation}
for almost everywhere $x\in\mathcal{S},$ where $K\in L_{loc}^{p}(\mathcal{S})$ for some $p>n$ and  $\Sigma\in L_{loc}^{1+\varepsilon}(\mathcal{S})$ for some $\varepsilon>0.$  Suppose additionally that $U_{\varepsilon_{1}}$ is a connected component of $f^{-1}(\mathbb{B}(y_{0},\varepsilon_{1})),$ which $\overline{U_{\varepsilon_{1}}}$ is compactly contained in $\Omega,$  then we have
\begin{equation}\label{shdhjsa}
{\rm deg}(f,U_{\varepsilon_{1}})=\frac{1}{\omega_{n}\varepsilon_{1}^{n}}\int_{U_{\varepsilon_{1}}}J_{f}d\mathcal{H}^{n}.
\end{equation}
\end{lem}

\section{The proof of Reshetnyak's theorem for quasiregular values}\label{schjcsjb}
The main purpose of this section is to provide the proof of Reshetnyak's theorem for quasiregular values from generalized $n$-manifold with controlled geometry to Euclidean space, and our proof follows a similar line to the one given by Kangasniemi and  Onninen \cite{KO222}. Just as done by Kangasniemi and  Onninen in \cite[Lemma 5.2]{KO222}, we first give the following result concerning the non-negative degree on the components of $f^{-1}(\mathbb{B}(y_{0},\varepsilon_{1})).$  The key tool to prove Lemma \ref{ahkkdjk} is the Jacobian-degree formula in Lemma \ref{lomlmass}. Its proof is similar to that of \cite[Lemma 5.2]{KO222}, and is therefore omitted here.
\begin{lem}\label{ahkkdjk}
Suppose that $\Omega\subset\mathcal{S}$ is a bounded connected domain and $y_{0}\in\mathbb{R}^{n}$. Suppose that a nonconstant continuous map $f\in N_{loc}^{1,n}(\Omega,\mathbb{R}^{n})$ 
 satisfies the following inequality
\begin{equation}
\| {\rm ap}Df(x)\|^{n}\leq K(x)J_{f}(x)+\Sigma(x)\abs{f(x)-y_{0}}^{n}
\end{equation}
for almost everywhere $x\in\Omega,$ where $K\in L_{loc}^{p}(\mathcal{S})$ for some $p>n$ and  $\Sigma\in L_{loc}^{1+\varepsilon}(\mathcal{S})$ for some $\varepsilon>0.$   Suppose additionally that $U_{\varepsilon_{1}}$ is a connected component of $f^{-1}(\mathbb{B}(y_{0},\varepsilon_{1})),$ which $\overline{U_{\varepsilon_{1}}}$ is compactly contained in $\Omega.$ Then, there exists $C=C(n,K,\Sigma,\Omega)>0,$ such that 
${\rm deg}(f,U_{\varepsilon_{1}})\geq0$ if $\mathcal{H}^{n}(U_{\varepsilon_{1}})\leq C.$
\end{lem}
In \cite[Theorem 5.3]{Kir16}, Kirsil\"{a} showed that if a continuous map $f\in N_{loc}^{1,n}(\Omega,\mathbb{R}^{n})$ satisfies $J_{f}(x)\geq0$ for almost everywhere $x\in\Omega,$ then $f$ is sense-preserving; i.e. it holds  ${\rm deg}(y, f,D)>0$ for every $D\subset\subset\Omega$ and $y\in f(D)\setminus f(\partial D).$ However, this result in \cite[Theorem 5.3]{Kir16} is not available for our case.
This is because it allows the Jacobi determinant $J_{f}$ of a map $f$ with value of finite distortion at $y_{0}$ to be negative on its domain.  While, we still can show that there exists a small 
$\varepsilon_{1}>0,$ such that the degree of $f$ on $U_{\varepsilon_{1}}$ is always positive. Here,  $U_{\varepsilon_{1}}$ is a connected component of $f^{-1}(\mathbb{B}(y_{0},\varepsilon_{1})).$ 
This, in turn, tells us that the integral average of  $J_{f}$ over $U_{\varepsilon_{1}}$ is strictly positive. This is the following result which plays a key role in the proof of discreteness and locally positive index  for quasiregular values. The idea of the proof is based on \cite{KO222, Zho251}.
\begin{lem}\label{bgjxolohcn}
For a constant $K\geq1,$ we suppose that $\Omega\subset\mathcal{S}$ is a bounded connected domain and $y_{0}\in\mathbb{R}^{n}$. Suppose that a nonconstant continuous map $f\in N_{loc}^{1,n}(\Omega,\mathbb{R}^{n})$ 
 satisfies the following inequality
\begin{equation}\label{jjksxc}
\| {\rm ap}Df(x)\|^{n}\leq KJ_{f}(x)+\Sigma(x)\abs{f(x)-y_{0}}^{n}
\end{equation}
for almost everywhere $x\in\Omega,$ where  $\Sigma\in L_{loc}^{p}(\mathcal{S})$ for some $p>1.$   Suppose additionally that $U_{\varepsilon_{1}}$ is a connected component of $f^{-1}(\mathbb{B}(y_{0},\varepsilon_{1})),$ which $\overline{U_{\varepsilon_{1}}}$ is compactly contained in $\Omega.$ Then, we have
${\rm deg}(f,U_{\varepsilon_{1}})>0.$
\end{lem}
\begin{proof}
According to Lemma \ref{ahkkdjk}, there exists $\varepsilon_{1}>0$ such that $U_{\varepsilon_{1}}\subset\subset\Omega$ is a connected component of $f^{-1}(\mathbb{B}(y_{0},\varepsilon_{1}))$ satisfying that 
$\mathcal{H}^{n}(U_{\varepsilon_{1}})\leq C$ and thus ${\rm deg}(f,U_{\varepsilon_{1}})\geq0.$ Next, we will show that  ${\rm deg}(f,U_{\varepsilon_{1}})=0$ is impossible, and thus finishes the proof of 
Lemma \ref{bgjxolohcn}. We verify this by contradiction. We split the proof of this lemma into three parts.

$\mathbf{Part}\;\mathbf{I}$: Showing that if ${\rm deg}(f,U_{\varepsilon_{1}})=0,$ then we have 
$$\int_{U_{\varepsilon_{1}}}\frac{\|apDf\|^{n}}{\abs{f-y_{0}}^{n}}d\mathcal{H}^{n}<\infty\;\;\;\;{\rm and }\;\;\;\;\int_{U_{\varepsilon_{1}}}\frac{J_{f}}{\abs{f-y_{0}}^{n}}d\mathcal{H}^{n}=0.$$
We first show that if
${\rm deg}(f,U_{\varepsilon_{1}})=0,$ then  we have 
\begin{equation}\label{dwshiksz}
\int_{U_{\varepsilon_{1}}\cap f^{-1}(\mathbb{B}(y_{0},r))}J_{f}d\mathcal{H}^{n}=0
\end{equation}
for almost every $r\in(0,\varepsilon_{1}).$ Without loss of generality, we assume that $f\in N^{1,n}(U_{\varepsilon_{1}},\mathbb{R}^{n})$ and  $\Sigma\in L^{p}(\varepsilon_{1}).$
Let $J_{f}=J_{f}^{+}-J_{f}^{-}$ and $U_{r}=U_{\varepsilon_{1}}\cap f^{-1}(\mathbb{B}(y_{0},r)).$ Then, according to Lemma \ref{lomlmass}, we see that if ${\rm deg}(f,U_{\varepsilon_{1}})=0,$ then we have $\int_{U_{\varepsilon_{1}}}J_{f}d\mathcal{H}^{n}=0.$ Since $U_{\varepsilon_{1}}\cap f^{-1}(\mathbb{B}(y_{0},r))$
is a disjoint union components of $f^{-1}(\mathbb{B}(y_{0},r))$ and $\mathcal{H}^{n}(U_{\varepsilon_{1}}\cap f^{-1}(\mathbb{B}(y_{0},r)))\leq\mathcal{H}^{n}(U_{\varepsilon_{1}})\leq C,$ 
we deduce from Lemma \ref{ahkkdjk} that $0\leq{\rm deg}(f,U_{\varepsilon_{1}}\cap f^{-1}(\mathbb{B}(y_{0},r)))\leq{\rm deg}(f,U_{\varepsilon_{1}})=0.$ 
 Hence, by using of formula (\ref{shdhjsa}), we deduce the desired claim of equation (\ref{dwshiksz}).
Now, by virtue of Lemma \ref{lomlmass}, we obtain
 \begin{equation}\label{ed2rf2}
 \int_{U_{r}}J_{f}^{+}d\mathcal{H}^{n}=\int_{U_{r}}J_{f}^{-}d\mathcal{H}^{n}.
 \end{equation}
 Since from (\ref{jjksxc}), we have that 
$$\int_{U_{\varepsilon_{1}}}\frac{J_{f}^{-}}{\abs{f-y_{0}}^{n}}d\mathcal{H}^{n}\leq\int_{U_{\varepsilon_{1}}}\frac{\Sigma}{K}\aleq\left\|\Sigma\right\|_{L^{p}(U_{\varepsilon_{1}})}\cdot (\mathcal{H}^{n}(U_{\varepsilon_{1}}))^{\frac{p-1}{p}}<\infty.$$
Hence, by multiplying $nr^{n}$ in (\ref{ed2rf2}) and integrating, we obtain through the Fubini-Tonelli Theorem that 
$$\int_{U_{\varepsilon_{1}}}\left(\frac{J_{f}^{+}}{\abs{f-y_{0}}^{n}}-\frac{J_{f}^{+}}{\varepsilon_{1}}\right)d\mathcal{H}^{n}=\int_{U_{\varepsilon_{1}}}\left(\frac{J_{f}^{-}}{\abs{f-y_{0}}^{n}}-
\frac{J_{f}^{-}}{\varepsilon_{1}}\right)d\mathcal{H}^{n},$$
which yields that  
$$\int_{U_{\varepsilon_{1}}}\frac{J_{f}}{\abs{f-y_{0}}^{n}}d\mathcal{H}^{n}=0.$$ 
Now, by the definition of quasiregular value, we deduce 
\begin{equation}
\begin{aligned}
\int_{U_{\varepsilon_{1}}}\frac{\|apDf\|^{n}}{\abs{f-y_{0}}^{n}}d\mathcal{H}^{n}\leq&K\int_{U_{\varepsilon_{1}}}\frac{J_{f}}{\abs{f-y_{0}}^{n}}d\mathcal{H}^{n}+
K\int_{U_{\varepsilon_{1}}}\Sigma\\
=&K\int_{U_{\varepsilon_{1}}}\Sigma\aleq\left\|\Sigma\right\|_{L^{p}(U_{\varepsilon_{1}})}\cdot (\mathcal{H}^{n}(U_{\varepsilon_{1}}))^{\frac{p-1}{p}}<\infty.
\end{aligned}
\end{equation}
This finishes the proof of $\mathbf{Part}\;\mathbf{I}.$

$\mathbf{Part}\;\mathbf{II}.$ Let $u=\log\vert f-y_{0}\vert,$ then we show that there exists $Q>n,$ such that $u\in N_{loc}^{1,Q}(U_{\varepsilon_{1}}),$ and thus  
 $u\in L_{loc}^{\infty}(U_{\varepsilon_{1}}).$

Without loss of generality, we assume that $y_{0}=0.$ We first show that  $g_{u}\in L_{loc}^{n}(U_{\varepsilon_{1}}).$  For any $\lambda>0,$ we denote $\vert f\vert_{\lambda}=\max\{\vert f\vert,\lambda\}.$ Then we have $\vert f\vert_{\lambda}\in N_{loc}^{1,n}(\Omega)$ as $f\in N_{loc}^{1,n}(\Omega).$
Hence, by the locally Lipschitz continuous of $\log(\cdot),$ we deduce that $u_{\lambda}:=\log\vert f\vert_{\lambda}\in N_{loc}^{1,n}(\Omega).$ 
Hence, when $\vert f\vert\geq\lambda>0,$ we see from \cite[Proposition 2.17]{BB11}  that
$$g_{u_{\lambda}}=\frac{g_{\vert f\vert_{\lambda}}}{\vert f\vert_{\lambda}}\leq\frac{\|apDf\|}{\vert f\vert}<\infty$$ is 
 a minimal $n$-weak gradient of $u_{\lambda}.$  Hence, when $\vert f\vert\geq\lambda>0,$ we get 
 $$\int_{\Omega_{1}}g_{u_{\lambda}}^{n}d\mathcal{H}^{n}\leq\int_{U_{\varepsilon_{1}}}\frac{\|apDf\|^{n}}{\abs{f}^{n}}d\mathcal{H}^{n}\aleq\left\|\Sigma\right\|_{L^{p}(U_{\varepsilon_{1}})}\cdot (\mathcal{H}^{n}(U_{\varepsilon_{1}}))^{\frac{p-1}{p}}<\infty.$$
Now, by dominated convergence theorem, we deduce that 
$$\int_{\Omega_{1}}g_{u}^{n}d\mathcal{H}^{n}\leq\int_{U_{\varepsilon_{1}}}\frac{\|apDf\|^{n}}{\abs{f}^{n}}d\mathcal{H}^{n}\aleq\left\|\Sigma\right\|_{L^{p}(U_{\varepsilon_{1}})}\cdot (\mathcal{H}^{n}(U_{\varepsilon_{1}}))^{\frac{p-1}{p}}<\infty,$$
where $\Omega_{1}\subset\subset U_{\varepsilon_{1}}.$ So we have $g_{u}\in L_{loc}^{n}(U_{\varepsilon_{1}}).$

Next, we show that there exists $\beta>1,$ such that $g_{u}^{n}\in L_{loc}^{\beta}(U_{\varepsilon_{1}}).$ To do this, we first to show that $u\in L_{loc}^{1}(\Omega).$

Next, we show that $g_{u}^{n}\in L_{loc}^{\beta}(U_{\varepsilon_{1}}),$ where  $\beta>1.$ We do this by using of the  Caccioppoli-type estimate from Lemma \ref{shksujck}.
For every fixed $x\in U_{\varepsilon_{1}},$ we suppose that $0<r<2r<{\rm dist}(x,\partial U_{\varepsilon_{1}}),$ and suppose that the Lipschitz mapping $\eta\vert_{B(x,r)}=1$ and 
$(\vert\nabla\eta\vert)\vert_{B(x,2r)}\leq C/r.$ In addition, we suppose that 
\begin{equation}
  \psi_{a,\varepsilon}(t)=\left\{
\begin{aligned}
&\frac{\log\sqrt{t+\varepsilon}-c}{t^{\frac{n}{2}}},      &   &    {\rm if}\;{t\geq a^{2}},\\
&\frac{\log\sqrt{a^{2}+\varepsilon}-c}{a^{n}},       &      &{\rm if}\; {t\leq a^{2}.}
\end{aligned} \right.
\end{equation}
Hence, by using of Lemma \ref{shksujck}, we obtain that 
\begin{equation}\label{shbjzls}
\begin{aligned}
&\abs{\int_{B(x,2r)\cap\{\vert f\vert>a\}}\frac{\eta J_{f}}{\vert f\vert^{n-2}(\varepsilon+\vert f\vert^{2})}d\mathcal{H}^{n}}\\
\leq&\left\vert\log\sqrt{a^{2}+\varepsilon}-c\right\vert\int_{B(x,2r)\cap\{\vert f\vert<a\}}\frac{\vert J_{f}\vert}{a^{n}}d\mathcal{H}^{n}\\
+& 
\frac{C}{r}\int_{B(x,2r)\cap\{\vert f\vert>a\}}\frac{\vert\log\sqrt{\varepsilon+\vert f\vert^{2}}-c\vert\cdot\|apDf\|^{n-1}}{\vert f\vert^{n-1}}d\mathcal{H}^{n} \\
+& 
\frac{C}{r}\left\vert\log\sqrt{a^{2}+\varepsilon}-c\right\vert\int_{B(x,2r)\cap\{\vert f\vert<a\}}\frac{\vert f\vert\cdot\|apDf\|^{n-1}}{a^{n}}d\mathcal{H}^{n}.
\end{aligned}
\end{equation}
Since 
\begin{equation}
\begin{aligned}
\int_{B(x,2r)\cap\{\vert f\vert<a\}}\frac{\vert J_{f}\vert}{a^{n}}d\mathcal{H}^{n}\leq&\int_{B(x,2r)\cap\{\vert f\vert<a\}}\frac{\vert J_{f}\vert}{\vert f\vert^{n}}d\mathcal{H}^{n}\\
\leq&\int_{B(x,2r)\cap\{\vert f\vert<a\}}\frac{\vert apDf\vert^{n}}{\vert f\vert^{n}}d\mathcal{H}^{n},
\end{aligned}
\end{equation}
 $\vert apDf\vert^{n}/\vert f\vert^{n}\in L_{loc}^{1}(U_{\varepsilon_{1}})$ and $\mathcal{H}^{n}(f^{-1}(0))=0,$ we deduce that 
 $$\left\vert\log\sqrt{a^{2}+\varepsilon}-c\right\vert\int_{B(x,2r)\cap\{\vert f\vert<a\}}\frac{\vert J_{f}\vert}{a^{n}}d\mathcal{H}^{n}\rightarrow0\;\;{\rm as}\;\;a\rightarrow0.$$
Again, as $\vert apDf\vert^{n}/\vert f\vert^{n}\in L_{loc}^{1}(U_{\varepsilon_{1}}),$  $\mathcal{H}^{n}(f^{-1}(0))=0$ and 
\begin{equation}
\begin{aligned}
&\int_{B(x,2r)\cap\{\vert f\vert<a\}}\frac{\vert f\vert\cdot\|apDf\|^{n-1}}{a^{n}}d\mathcal{H}^{n}\leq\int_{B(x,2r)\cap\{\vert f\vert<a\}}\frac{\|apDf\|^{n-1}}{\vert f\vert^{n-1}}d\mathcal{H}^{n}\\
\leq&\left(\int_{B(x,2r)\cap\{\vert f\vert<a\}}\frac{\vert apDf\vert^{n}}{\vert f\vert^{n}}d\mathcal{H}^{n}\right)^{\frac{n-1}{n}}\left(\int_{B(x,2r)\cap\{\vert f\vert<a\}}\right)^{\frac{1}{n}}\\
\leq&\left(\int_{B(x,2r)\cap\{\vert f\vert<a\}}\frac{\vert apDf\vert^{n}}{\vert f\vert^{n}}d\mathcal{H}^{n}\right)^{\frac{n-1}{n}}\left(\mathcal{H}^{n}(U_{\varepsilon_{1}})\right)^{\frac{1}{n}},
\end{aligned}
\end{equation}
we obtain that 
 $$\frac{C}{r}\left\vert\log\sqrt{a^{2}+\varepsilon}-c\right\vert\int_{B(x,2r)\cap\{\vert f\vert<a\}}\frac{\vert f\vert\cdot\|apDf\|^{n-1}}{a^{n}}d\mathcal{H}^{n}\rightarrow0\;\;{\rm as}\;\;a\rightarrow0.$$
Since 
$$\frac{\vert\log\sqrt{\varepsilon+\vert f\vert^{2}}-c\vert\cdot\|apDf\|^{n-1}}{\vert f\vert^{n-1}}\aleq\frac{(\vert\log\vert f\vert\vert+C)\cdot\|apDf\|^{n-1}}{\vert f\vert^{n-1}}$$
for some $C>0,$  $\vert apDf\vert^{n-1}/\vert f\vert^{n-1}\in L_{loc}^{n/(n-1)}(U_{\varepsilon_{1}})$ and $\log\vert f\vert\in L_{loc}^{n}(U_{\varepsilon_{1}}),$ we get by dominated convergence theorem as $a\rightarrow0^{+}$ in
(\ref{shbjzls}) that 
\begin{equation}\label{shdsbsbjzlssdvs}
\begin{aligned}
\abs{\int_{B(x,2r)}\frac{\eta J_{f}}{\vert f\vert^{n-2}(\varepsilon+\vert f\vert^{2})}d\mathcal{H}^{n}}
\leq\frac{C}{r}\int_{B(x,2r)}\frac{\vert\log\sqrt{\varepsilon+\vert f\vert^{2}}-c\vert\cdot\|apDf\|^{n-1}}{\vert f\vert^{n-1}}d\mathcal{H}^{n}.
\end{aligned}
\end{equation}
 We use dominated convergence theorem again in (\ref{shdsbsbjzlssdvs}) as $\varepsilon\rightarrow0^{+},$ then we have  
\begin{equation}\label{shadvfbgnbjzlssdvs}
\begin{aligned}
\abs{\int_{B(x,2r)}\frac{\eta J_{f}}{\vert f\vert^{n}}d\mathcal{H}^{n}}
\leq\frac{C}{r}\int_{B(x,2r)}\frac{\vert\log\vert f\vert-c\vert\cdot\|apDf\|^{n-1}}{\vert f\vert^{n-1}}d\mathcal{H}^{n}.
\end{aligned}
\end{equation}
Now, by using of H\"{o}lder inequality and Sobolev-Poincar\'{e} inequality (\ref{asddfvddjaadlla}), we deduce that 
\begin{equation}\label{dhiidha}
\begin{aligned}
&\int_{B(x,2r)}\frac{\|apDf\|^{n-1}}{\vert f\vert^{n-1}}\cdot \vert\log\vert f\vert-c\vert d\mathcal{H}^{n}\\
\leq&\left(\int_{B(x,2r)}\left(\frac{\vert apDf\vert^{n}}{\vert f\vert^{n}}\right)^{\frac{n}{n+1}}d\mathcal{H}^{n}\right)^{\frac{n^{2}-1}{n^{2}}}
\left(\int_{B(x,2r)}\vert\log\vert f\vert-c\vert^{n^{2}}\right)^{\frac{1}{n^{2}}}\\
\aleq&\left(\int_{B(x,2r)}\left(\frac{\vert apDf\vert^{n}}{\vert f\vert^{n}}\right)^{\frac{n}{n+1}}d\mathcal{H}^{n}\right)^{\frac{n^{2}-1}{n^{2}}}\left(\int_{B(x,2\lambda r)}\left(\frac{\vert apDf\vert^{n}}{\vert f\vert^{n}}\right)^{\frac{n}{n+1}}d\mathcal{H}^{n}\right)^{\frac{n+1}{n^{2}}}\\
\aleq&\left(\int_{B(x,2\lambda r)}\left(\frac{\vert apDf\vert^{n}}{\vert f\vert^{n}}\right)^{\frac{n}{n+1}}d\mathcal{H}^{n}\right)^{\frac{n+1}{n}}.
\end{aligned}
\end{equation}
Combing inequalities (\ref{shadvfbgnbjzlssdvs}) and (\ref{dhiidha}), we have
\begin{equation}\label{1}
\begin{aligned}
\abs{\int_{B(x,2r)}\frac{\eta J_{f}}{\vert f\vert^{n}}d\mathcal{H}^{n}}
\aleq\frac{1}{r}\left(\int_{B(x,2\lambda r)}\left(\frac{\vert apDf\vert^{n}}{\vert f\vert^{n}}\right)^{\frac{n}{n+1}}d\mathcal{H}^{n}\right)^{\frac{n+1}{n}},
\end{aligned}
\end{equation}
which implies that 
\begin{equation}\label{123dhiidha}
\begin{aligned}
&\int_{B(x,r)}\frac{\|apDf\|^{n}}{\vert f\vert^{n}}d\mathcal{H}^{n}\leq\int_{B(x,2r)}\eta\frac{\|apDf\|^{n}}{\vert f\vert^{n}}d\mathcal{H}^{n}\\
\leq&K\abs{\int_{B(x,2r)}\frac{\eta J_{f}}{\vert f\vert^{n}}d\mathcal{H}^{n}}+\int_{B(x,2r)}\Sigma d\mathcal{H}^{n}\\
\aleq&\frac{1}{r}\left(\int_{B(x,2\lambda r)}\left(\frac{\vert apDf\vert^{n}}{\vert f\vert^{n}}\right)^{\frac{n}{n+1}}d\mathcal{H}^{n}\right)^{\frac{n+1}{n}}+
\int_{B(x,2\lambda r)}\Sigma d\mathcal{H}^{n}.\\
\end{aligned}
\end{equation}
Namely,
\begin{equation}\label{e14}
\begin{aligned}
&\int_{B(x,r)}\!\!\!\!\!\!\!\!\!\!\!\!\!\!\!\!\!\!\!\!\; {}-{} \,\,\,\,\,\,\,\,\frac{\|apDf\|^{n}}{\vert f\vert^{n}}d\mathcal{H}^{n}\aleq \left(\int_{B(x,2\lambda r)}\!\!\!\!\!\!\!\!\!\!\!\!\!\!\!\!\!\!\!\!\!\!\!\!\!\; {}-{} \,\,\,\,\,\,\,\,\,\,\,\;\left(\frac{\vert apDf\vert^{n}}{\vert f\vert^{n}}\right)^{\frac{n}{n+1}}d\mathcal{H}^{n}\right)^{\frac{n+1}{n}}\\
&\,\,\,\,\,\,\,\,\,\,\,\,\,\,\,\,\,\,\,\,\,\,\,\,\,\,\,\,\,\,\,\,\,\,\,\,\,\,\,\,\,\,\,\,\,\,\,\,\,\,\,\,\,\,\,\,\,\,\,\,\,\,\,\,\,\,\,\,\,\,\,\,\,\,\,\,\,\,\,\,\,\,\,\,\,\,\,\,\,\,\,\,\,\,\,\,\,\,\,
\,\,\,\,\,\,\,\,\,\,\,\,+
\int_{B(x,2\lambda r)}\!\!\!\!\!\!\!\!\!\!\!\!\!\!\!\!\!\!\!\!\!\!\!\!\!\; {}-{} \,\,\,\,\,\,\,\,\,\,\,\;\;\Sigma d\mathcal{H}^{n}.
\end{aligned}
\end{equation}
If we denote $q_{1}=(n+1)/n, r_{0}=pq_{1}>q_{1}, g=(\|apDf\|^{n}/\vert f\vert^{n})^{n/(n+1)}$ and $f=\Sigma^{n/(n+1)},$
then we have $f\in L_{loc}^{r_{0}}(U_{\varepsilon_{1}})$ and 
\begin{equation}
\begin{aligned}
&\int_{B(x,r)}\!\!\!\!\!\!\!\!\!\!\!\!\!\!\!\!\!\!\!\!\; {}-{} \,\,\,\,\,\,\,\,\,\,\,g^{q_{1}}\aleq\left[\left(\int_{B(x,2r)}\!\!\!\!\!\!\!\!\!\!\!\!\!\!\!\!\!\!\!\!\!\!\!\; {}-{} \,\,\,\,\,\,\,\,\,\,\,\,\,\,g\right)^{q_{1}}+
\int_{B(x,2r)}\!\!\!\!\!\!\!\!\!\!\!\!\!\!\!\!\!\!\!\!\!\!\!\; {}-{} \,\,\,\,\,\,\,\,\,\,\,\,\,f^{q_{1}}\right].\\
\end{aligned}
\end{equation}
Hence, by a Gehring's lemma on metric measure space \cite[Theorem 3.3]{Gol05}, we see that   there exists $\beta>1,$ such that $\|apDf\|^{n}/\vert f\vert^{n}\in L_{loc}^{\beta}(U_{\varepsilon_{1}}).$  So we have $g_{u}^{n}\in L_{loc}^{\beta}(U_{\varepsilon_{1}}),$ and thus $u\in N_{loc}^{1,\beta n}(U_{\varepsilon_{1}}).$ Now, by virtue of \cite[Theorem 5.1]{HK00}, we obtain that $u$ is local H\"{o}lder on $U_{\varepsilon_{1}},$ which yields that $u\in L_{loc}^{\infty}(U_{\varepsilon_{1}}).$ This finishes the proof.

$\mathbf{Part}\;\mathbf{III}.$ By showing that $u=\log\vert f-y_{0}\vert$ is unbounded near $x\in f^{-1}\{y_{0}\}.$ This is because $u$ is continuous and thus we get 
$$\lim_{x\rightarrow f^{-1}\{y_{0}\}}u(x)=\lim_{x\rightarrow f^{-1}\{y_{0}\}}\log\vert x-y_{0}\vert=\lim_{f(x)\rightarrow y_{0}}\log\vert f(x)-y_{0}\vert=\infty.$$

Now, we finish the proof of Lemma \ref{bgjxolohcn}.   We assume that  ${\rm deg}(f,U_{\varepsilon_{1}})=0,$ then  from $\mathbf{Part}\;\mathbf{II},$ we see that $u\in L_{loc}^{\infty}(U_{\varepsilon_{1}}).$
This contradicts the conclusion of  $\mathbf{Part}\;\mathbf{III}.$ This completes proof of Lemma \ref{bgjxolohcn}. 
\end{proof}

\subsection{Discrete type result for quasiregular values} 
In this section, we give the proof of following result. The idea of the proof is based on  \cite{TS62,Kir14}
\begin{lem}\label{th11000dsvfffbvd}
Let $K\geq1$ be a constant, $\Omega\subset\mathcal{S}$ be a domain and $y_{0}\in\mathbb{R}^{n}$. Suppose that $f$ belongs to  $N_{loc}^{1,n}(\Omega,\mathbb{R}^{n})\cap \mathcal{C}(\Omega,\mathbb{R}^{n})$ and satisfies the following inequality
\begin{equation}
\| {\rm ap}Df(x)\|^{n}\leq KJ_{f}(x)+\Sigma(x)\abs{f(x)-y_{0}}^{n}
\end{equation}
for almost everywhere $x\in\Omega.$ If $\Sigma\in L_{loc}^{p}(\Omega)$ for some $p>1,$  then $f^{-1}\{y_{0}\}$ is discrete.
\end{lem}
\begin{proof}
Without loss of generality, we assume that $\Omega$ is a bounded domain, $f\in N^{1,n}(\Omega,\mathbb{R}^{n})\cap \mathcal{C}(\Omega,\mathbb{R}^{n})$ and $\Sigma\in L^{p}(\Omega).$
We assume to the contrary that  $f^{-1}\{y_{0}\}$ is not discrete. Then, there exists $x_{0}$ such that for every neighborhood $U_{0}$ of $x_{0},$ there always exists a $x\in f^{-1}\{y_{0}\}\setminus\{x_{0}\}.$
Now, by using of Lemma \ref{bgjxolohcn}, we select $\varepsilon_{1}>0$ such that if  $U_{\varepsilon_{1}}$ is a $x_{0}$-connected component of $f^{-1}(\mathbb{B}(y_{0},\varepsilon_{1})),$ which $\overline{U_{\varepsilon_{1}}}$ is compactly contained in $\Omega,$ then we have ${\rm deg}(f,U_{\varepsilon_{1}})>0.$  If we denote $N={\rm deg}(f,U_{\varepsilon_{1}}),$ then by formula (\ref{shdhjsa})
and  $J_{f}\in L^{1}(\Omega),$ we also see that ${\rm deg}(f,U_{\varepsilon_{1}})$ is finite. Now,  as  $f^{-1}\{y_{0}\}$ is not discrete,  we can find $x_{1},x_{2},\ldots,x_{N+1}\in U_{\varepsilon_{1}}\cap f^{-1}\{y_{0}\}$ and radii $r_{i},$ such that $B(x_{i},r_{i})\cap B(x_{j},r_{j})=\emptyset$ for every different $i,j\in\{1,2,\ldots,N+1\}$ and $y_{0}\notin\bigcup_{i=1}^{N+1}f(\partial B(x_{i},r_{i})).$
We denote $B_{j}$ as the $x_{j}$-connected component of those balls, then $y_{0}\notin\bigcup_{i=1}^{N+1}f(\partial B{i}).$ Now, since $U_{\varepsilon_{1}}\cap f^{-1}\{y_{0}\}\setminus\bigcup_{i=1}^{N+1}B{i}$
is compact, we chose a finite subcover $\{U_{i}\}_{i=1}^{M},$
such that $N={\rm deg}(f,U_{\varepsilon_{1}})=\sum_{i=1}^{N+1}{\rm deg}(f,y_{0},B_{i})+\sum_{i=1}^{M}{\rm deg}(f,y_{0},U_{i})\geq N+1.$ This is contradiction and thus This finishes the proof of Lemma
Lemma \ref{th11000dsvfffbvd}.
\end{proof}

\subsection{Locally positive index for values of finite distortion} 
This section is devoting to proving the following conclusion.
\begin{lem}\label{th11000dsvfffbvdasxcscs}
Let $K\geq1$ be a constant, $\Omega\subset\mathcal{S}$ be a domain and $y_{0}\in\mathbb{R}^{n}$. Suppose that $f$ belongs to  $N_{loc}^{1,n}(\Omega,\mathbb{R}^{n})\cap \mathcal{C}(\Omega,\mathbb{R}^{n})$ and satisfies the following inequality
\begin{equation}
\| {\rm ap}Df(x)\|^{n}\leq KJ_{f}(x)+\Sigma(x)\abs{f(x)-y_{0}}^{n}
\end{equation}
for almost everywhere $x\in\Omega.$ If $\Sigma\in L_{loc}^{p}(\Omega)$ for some $p>1,$
then the local index $i(x,f)$ is positive for every $x\in f^{-1}\{y_{0}\}.$
\end{lem}
\begin{proof}
Without loss of generality, we assume that $\Omega$ is a bounded domain, $f\in N^{1,n}(\Omega,\mathbb{R}^{n})\cap \mathcal{C}(\Omega,\mathbb{R}^{n})$ and $\Sigma\in L^{p}(\Omega).$
Then, according to  Lemma \ref{bgjxolohcn}, we see that for  every $x\in f^{-1}\{y_{0}\},$ there exists $\varepsilon_{1}>0,$ such that $U_{\varepsilon_{1}}$ is a connected $x$-component of $f^{-1}(\mathbb{B}(y_{0},\varepsilon_{1})),$ which $\overline{U_{\varepsilon_{1}}}$ is compactly contained in $\Omega$ and satisfying that
${\rm deg}(f,U_{\varepsilon_{1}})>0.$ According to Lemma \ref{th11000dsvfffbvd}, we can make appropriate adjustment of $\varepsilon_{1},$ such that that $i(x,f)={\rm deg}(f,U_{\varepsilon_{1}}),$ and thus is positive. This finishes the proof of Lemma 
\ref{th11000dsvfffbvdasxcscs}.
\end{proof}

\subsection{Local openness for quasiregular values} 
The proof of following lemma is based on \cite[Lemma 6.3]{KO222}
\begin{lem}\label{txcsh11000dsvfffbvdasxcscc ssc s}
Let $K\geq1$ be a constant, $\Omega\subset\mathcal{S}$ be a domain and $y_{0}\in\mathbb{R}^{n}$. Suppose that $f$ belongs to  $N_{loc}^{1,n}(\Omega,\mathbb{R}^{n})\cap \mathcal{C}(\Omega,\mathbb{R}^{n})$ and satisfies the following inequality
\begin{equation}
\| {\rm ap}Df(x)\|^{n}\leq KJ_{f}(x)+\Sigma(x)\abs{f(x)-y_{0}}^{n}
\end{equation}
for almost everywhere $x\in\Omega.$ If $\Sigma\in L_{loc}^{p}(\Omega)$ for some $p>1,$ then for every neighborhood $V$ of every  $x\in f^{-1}\{y_{0}\},$  there is valid for $y_{0}\in {\rm int}f(V).$
\end{lem}
\begin{proof}
Suppose that there is a $x_{0}\in f^{-1}\{y_{0}\},$ such that there exists an open set $U_{0}$ of $x_{0}$ satisfying that $\mathbb{B}(y_{0},\varepsilon)\setminus f(U_{0})\neq\emptyset$ for all small $\varepsilon>0.$

We chose a $\varepsilon_{0}>0,$ such that $U_{\varepsilon_{0}}\subset\subset\Omega$ and satisfies $\overline{U_{\varepsilon_{0}}}\cap f^{-1}\{y_{0}\}=x_{0}$ and ${\rm deg}(f,U_{\varepsilon_{0}})>0.$
Now, for any $\varepsilon<\varepsilon_{0},$ we see that if denote $U_{\varepsilon}$ is the $x_{0}$-connected component of $f^{-1}(\mathbb{B}(y_{0},\varepsilon)),$ then we have 
$\overline{U_{\varepsilon}}\cap f^{-1}\{y_{0}\}=x_{0}$ and ${\rm deg}(f,U_{\varepsilon})>0.$ So by Lemma \ref{th11000dsvfffbvdasxcscs}, we have that  ${\rm deg}(y_{0}, f,U_{\varepsilon})>0$ for all 
$y_{0}\in\mathbb{B}(y_{0},\varepsilon).$ According to \cite[Lemma 2.7]{Kir16}, we have that $f\in f(U_{\varepsilon}).$ Hence, we have $\mathbb{B}(y_{0},\varepsilon)\subset f(U_{\varepsilon}).$
In addition, it is obvious that  $ f(U_{\varepsilon})\subset\mathbb{B}(y_{0},\varepsilon).$ Therefore, we have $f(U_{\varepsilon})=\mathbb{B}(y_{0},\varepsilon)$ for all
$\varepsilon<\varepsilon_{0}.$

From above, we see that $\mathbb{B}(y_{0},\varepsilon)\setminus f(U_{0})=f(U_{\varepsilon})\setminus f(U_{0})\neq\emptyset$ for all $\varepsilon<\varepsilon_{0}.$
So $U_{\varepsilon}\setminus U_{0}\neq\emptyset$ and thus $\overline{U_{\varepsilon}}\setminus U_{0}\neq\emptyset$ for all $\varepsilon<\varepsilon_{0}.$ This is contradiction since
$\bigcap_{\varepsilon<\varepsilon_{0}}\overline{U_{\varepsilon}}\setminus U_{0}=x_{0}\setminus U_{0}=\emptyset.$ This finishes the proof of Lemma \ref{txcsh11000dsvfffbvdasxcscc ssc s}.
\end{proof}



\section{Acknowledge}
This work was supported by Guangdong Basic and Applied Basic Research Foundation (No. 2022A1515110967 and No. 2023A1515011809).



\end{document}